\documentclass[aos,preprint]{imsart}
\setattribute{journal}{name}{}

\RequirePackage[OT1]{fontenc}
\RequirePackage{amsthm,amsmath,amssymb}
\RequirePackage[numbers]{natbib}
\RequirePackage[colorlinks,citecolor=blue,urlcolor=blue]{hyperref}
\RequirePackage[dvips]{graphicx}

\graphicspath{{./figures/}}

\usepackage[normalem]{ulem} % to use \sout
\def\R{\mathbb{R}}
\def\E{\mathbb{E}}
\def\supp{\mathrm{supp}}
\def\Pr{\mathrm{Prob}}
\def\sign{\mathrm{sign}}
\def\ker{\mathrm{ker}}
\def\tr{\mathrm{tr}}
\def\im{\mathrm{im}}
\def\NN{\mathcal{N}}

\def\K{\mathcal{K}}
\def\<{\langle}
\def\>{\rangle}

\numberwithin{equation}{section}
\theoremstyle{plain}
\newtheorem{thm}{Theorem}[section]

\newtheorem{cor}{Corollary}[section]
\newtheorem{lem}{Lemma}[section]
\newtheorem{rem}{Remark}[section]
\begin{document}

\begin{frontmatter}
\title{Sparse Recovery via Differential Inclusions}%\protect\thanksref{T1}}
\runtitle{Sparse Recovery via Differential Inclusions}
%\thankstext{T1}{Footnote to the title with the `thankstext' command.}

\begin{aug}
\author{\fnms{Stanley} \snm{Osher}\thanksref{m1}\ead[label=e1]{sjo@math.ucla.edu}},
\author{\fnms{Feng} \snm{Ruan}\thanksref{m2,m3}\ead[label=e2]{fengruan@stanford.edu}},
\author{\fnms{Jiechao} \snm{Xiong}\thanksref{m2}\ead[label=e3]{xiongjiechao@pku.edu.cn}},
\author{\fnms{Yuan} \snm{Yao}\thanksref{m2}\ead[label=e4]{yuany@math.pku.edu.cn}}
\and
\author{\fnms{Wotao} \snm{Yin}\thanksref{m1}
\ead[label=e5]{wotaoyin@math.ucla.edu}
\ead[label=u1,url]{http://www.math.pku.edu.cn/teachers/yaoy/}}

%\thankstext{t1}{Some comment}
%\thankstext{t2}{First supporter of the project}
%\thankstext{t3}{Second supporter of the project}
\runauthor{Osher, Ruan, Xiong, Yao and Yin}

\affiliation{University of California, Los Angeles\thanksmark{m1}, Peking University\thanksmark{m2}, and Stanford University\thanksmark{m3}}

\address{School of Mathematical Sciences\\
Peking University\\
Beijing, China 100871 \\
\printead{e2}\\
\phantom{E-mail:\ }\printead*{e3}\\
\phantom{E-mail:\ }\printead*{e4}\\
\printead{u1}}

\address{Department of Mathematics\\
University of California\\
Los Angels, CA 90095\\
\printead{e1}\\
\phantom{E-mail:\ }\printead*{e5}}
\end{aug}

\begin{abstract}
In this paper, we recover sparse signals from their noisy linear measurements by solving nonlinear differential inclusions, which is based on the notion of inverse scale space (ISS) developed in applied mathematics. Our goal here is to bring this idea to address a challenging problem in statistics, \emph{i.e.} finding the oracle estimator which is unbiased and sign-consistent using dynamics. We call our dynamics \emph{Bregman ISS} and \emph{Linearized Bregman ISS}. A well-known shortcoming of LASSO and any convex regularization approaches lies in the bias of estimators. However, we show that under proper conditions, there exists a bias-free and sign-consistent point on the solution paths of such dynamics, which corresponds to a signal that is the unbiased estimate of the true signal and whose entries have the same signs as those of the true signs, \emph{i.e.} the oracle estimator. Therefore, their solution paths are regularization paths better than the LASSO regularization path, since the points on the latter path are biased when sign-consistency is reached. We also show how to efficiently compute their solution paths in both continuous and discretized settings: the full solution paths can be exactly computed piece by piece, and a discretization leads to \emph{Linearized Bregman iteration}, which is a simple iterative thresholding rule and easy to parallelize. Theoretical guarantees such as sign-consistency and minimax optimal $l_2$-error bounds are established in both continuous and discrete settings for specific points on the paths. Early-stopping rules for identifying these points are given. The key treatment relies on the development of differential inequalities for differential inclusions and their discretizations, which extends the previous results and leads to exponentially fast recovering of sparse signals before selecting wrong ones.

%In this paper, we study sparse signal recovery via a differential inclusion approach whose discretization gives rise to the fast \emph{Linearized Bregman iteration} algorithm. Equipped with a proper early stopping rule, the solution path of such a differential inclusion returns an estimator that is both sign-consistent and bias-free, and thus better than the LASSO estimators. Sign-consistency and minimax optimal $l_2$-error bounds are established for early stopping regularization in both continuous dynamics and its discrete version. The key treatment relies on a development of differential inequalities for differential inclusions and their discretizations.
\end{abstract}

%\begin{keyword}[class=AMS]
%\kwd[Primary ]{60K35}
%\kwd{60K35}
%\kwd[; secondary ]{60K35}
%\end{keyword}

\footnote{Corresponding email: {\tt yuany@math.pku.edu.cn}, phone: 310-825-1758, fax: 310-206-6673 (c/o Stan Osher).}

\end{frontmatter}

\section{Introduction}
We study a dynamic approach to recover a sparse signal $\beta^*\in \R^p$ from its noisy linear measurements
\begin{equation}\label{yxbe}
y = X \beta^* + \epsilon.
\end{equation}
Here, $y\in \R^n$ is a measurement vector, $X=[x_1,\ldots,x_p]\in \R^{n\times p}$ is a measurement matrix, and $\epsilon\sim \NN(0,\sigma^2 I_n)$ is Gaussian noise. %with unknown variance $\sigma^2$.
We allow $n< p$ and assume that $\beta^*$ has  $s\le \min\{n,p\}$ nonzero components. For convenience,
let $S=\supp(\beta^*)$ and $T$ be its complement, i.e. $T=\{i:\beta_i^*=0\}$. $X_S$ denotes the submatrix of $X$ formed by the columns of $X$ in $S$, which are assumed to be \emph{linearly independent}. Similarly define $X_T$ so that $[X_S ~X_T]=X$.

Such a problem has been widely studied in applied mathematics \cite{bpdn}, engineering, and statistics \cite{lasso}, see for example surveys in \cite{Donoho06,CanWak08}. In these works, convex regularization or relaxation approach has been exploited to overcome the combinatorial explosion of searching the best sparse signals using subset least squares. However, it has been known since \cite{FanLi01} that all convex regularization approaches lead to biased estimators whose expectation does not meet the true signal, which motivates the exploration of using nonconvex regularization yet it may suffer from a computational hurdle of locating the global optima \cite{GJY11,GWYY15}. 

To address this dilemma between statistical accuracy and computational hurdle, in this paper we introduce some dynamics from the \emph{Inverse Scale Space} (ISS) method, which first appeared in the image restoration literature in \cite{burger2005nonlinear, BGOX06, BRH07, BFOS07,Burger08} and analyzed and implemented carefully in \cite{burger2013adaptive}. The name refers to the observation there that large-scale (image) features are recovered before small-scale ones. Our goal here is to show that such dynamics provides a surprisingly simple way to statistically accurate (unbiased and sign-consistent) estimator if equipped with a new type regularization -- early stopping. Our results also extend those early error analysis on ISS to statistical consistency, establishing model selection consistency as well as minimax optimal $l_2$ error bounds under comparable conditions to LASSO, \emph{etc}.
%We study two continuous time dynamics \emph{Bregman ISS} and \emph{Linearized Bregman ISS}, as well as the forward-Euler discretization of the latter,

The first one, called \emph{Bregman ISS} here, is given by the nonlinear differential inclusions:%(\textcolor{red}{solution existence and uniqueness?})
\begin{subequations}\label{eq:Bregman ISS}
\begin{align}
 \dot\rho_t&=\frac{1}{n} X^T(y-X\beta_t),\label{eq:Bregman ISSa}\\
 \rho_t &\in \partial\|\beta_t\|_1, \label{eq:Bregman ISSb}
\end{align}
\end{subequations}
where $t\ge 0$ is time, $\rho_t\in\R^p$ is assumed to be right continuously differentiable in $t$, $\dot\rho_t$ is the right derivative of $\rho_t$, and $\beta_t$ is assumed to be right continuous. The inclusion condition \eqref{eq:Bregman ISSb} restricts $\rho_t$ to a subgradient of $\ell_1$-norm  at $\beta_t$,  $t\ge 0$. The initial conditions are, typically,  $\rho_0=0$ and $\beta_0=0$. As it evolves, the component which reaches $|\rho_t(i)|=1$ enters into our selection $\beta_t(i)\neq 0$. Hence roughly speaking, the larger magnitude $X^T_i (y-X\beta_t)$ has, the faster the component is selected. In the ideal case, we hope the signals in $S$ are selected faster than non-signals in $T$, whose conditions will be our main concern in this paper. Under general conditions, we will see that a solution to \eqref{eq:Bregman ISS} exists and both $\rho_t$ and $X\beta_t$, $t\ge 0$, are unique. In addition, $\rho_t$ is piece-wise linear, and there exists a solution path $\beta_t$ that is piece-wise constant. The entire path can be computed at finitely many break points.

A damping version of the first one, called \emph{Linearized Bregman ISS}, has its solution path $\{\rho_t,\beta_t\}_{t\ge 0}$ governed by the nonlinear differential inclusions:
\begin{subequations}\label{eq:lb-iss}
\begin{align}
 \dot\rho_t + \frac{1}{\kappa}  \dot\beta_t &=\frac{1}{n} X^T(y-X\beta_t),\label{eq:lb-issa}\\
 \rho_t &\in \partial\|\beta_t\|_1, \label{eq:lb-issb}
\end{align}
\end{subequations}
where $\kappa>0$ is a constant.
Compared to  \eqref{eq:Bregman ISSa}, equation \eqref{eq:lb-issa} has the additional term $\frac{1}{\kappa}  \dot\beta$. As $\kappa\to \infty$, \eqref{eq:lb-iss} is reduced to \eqref{eq:Bregman ISS}, and the solution path of (\ref{eq:lb-iss}) may converge to that of (\ref{eq:Bregman ISS}) exponentially fast as $\kappa$ increases. We will see that  \eqref{eq:lb-iss} has a unique solution path $\rho_t$ and $\beta_t$, $t\ge 0$, which are both continuous for all $\kappa>0$. Alternatively, \eqref{eq:lb-iss} can be obtained as a differential inclusion replacing the $l_1$-norm in \eqref{eq:Bregman ISSb} by the Elastic Net \cite{ZH05} penalty $\|\beta_t\|_1+\frac{1}{\kappa}\|\beta_t\|_2^2$ which will be discussed later. 
%\textcolor{red}{It is impractical to compute the entire path of \eqref{eq:lb}.}

The discretizations of  \eqref{eq:Bregman ISS} and \eqref{eq:lb-iss} are known as Bregman Iteration (equation (3.7) of \cite{YODG08}) and Linearized Bregman Iteration (equations (5.19-20) of \cite{YODG08}), respectively. They were introduced in the literature of variational imaging and compressive sensing before  \eqref{eq:Bregman ISS} and \eqref{eq:lb-iss}. Through a change of variable, Bregman Iteration  becomes the iteration of the Augmented Lagrangian Method \cite{hestenes1969multiplier,powell1967method}. On the other hand, Linearized Bregman Iteration is a simple two-line iteration:
\begin{subequations}\label{eq:lb}
\begin{align}
\rho_{k+1} + \frac{1}{\kappa} \beta_{k+1} & =  \rho_{k} + \frac{1}{\kappa} \beta_{k} + \frac{\alpha_k}{n} X^T(y-X\beta_k),\\
\rho_k &\in \partial \|\beta_k\|_1,
\end{align}
\end{subequations}
which is evidently a forward Euler discretization to \eqref{eq:lb-iss}, where $\alpha_k>0$ is a step size.
Define $z_{k}=\rho_{k}+\frac{1}{\kappa} \beta_{k}$. Then \eqref{eq:lb} can be simplified to:\begin{subequations}\label{eq:lb2}
\begin{align}
z_{k+1}&= z_k +\frac{\alpha_k}{n} X^T(y-X\beta_k) \\
\beta_{k+1}&= \kappa \cdot \mathrm{shrink}(z_{k+1},1),\label{lb2b}
\end{align}
\end{subequations}
where the mapping $\mathrm{shrink}$ is defined component-wise  as
$$\mathrm{shrink}(z,\lambda):=\sign(z)\max\{|z|-\lambda,0\},\quad z,\lambda\in\R,~\lambda \ge0.$$
Note that $\mathrm{shrink}(z,\lambda)$ is the unique solution to the convex program:
$$\min_{x\in\R} |x|+\frac{1}{2\lambda}(x-z)^2, $$
which is called \emph{LASSO} in statistics literature \cite{lasso}.

\subsection{Motivations and contributions} \label{sec:motivation}
Our exposition is motivated by the fact that solution path $\{\beta_t\}_{t\ge 0}$ of the differential inclusion \eqref{eq:Bregman ISS} and the sequence  $\{\beta_k\}_{k\ge 0}$ of \eqref{eq:lb} are \emph{better} than the points on the LASSO regularization path. In particular, while LASSO regularization path is always biased, $\beta_t$ can be unbiased when the correct set of variables is reached. Here an estimator $\hat{\beta}$ is called \emph{biased} if its expectation does not equal to the true parameter $\beta^*$, i.e. $\E[\hat{\beta}]- \beta^*\neq 0$; otherwise it is called \emph{unbiased}. In fact, \cite{OBGXY05} observed in experiments that Bregman iterations may reduce bias in the context of Total Variation image denoising, analogous to the $l_1$ setting to be studied below.

To see this, consider the general LASSO problem \cite{lasso},
\begin{equation}\label{lass}
\min_\beta\lambda\|\beta\|_1+\frac{1}{2n}\|y-X\beta\|_2^2,
\end{equation}
where for the convenience of comparison we replace the regularization parameter $\lambda$ by $t=1/\lambda$ in the following equivalent form
\begin{equation}\label{lasso}
\min_\beta \|\beta\|_1+\frac{t}{2n}\| y - X\beta \|_2^2.
\end{equation}
%This formulation is a slight modification to the more common formulation in \eqref{lass}
Aside from the obvious relation $t=1/\lambda$, solution $\beta$ is piece-wise linear in $\lambda$ \cite{lars} though not  so in $t$. Despite this, $t$ will be convenient to our analysis by reflecting a nature of time evolution of the solution.

Since \eqref{lasso} is a convex program, $\hat{\beta}_t$ is a solution to \eqref{lasso} if and only if it obeys the first-order optimality conditions
\begin{subequations}\label{eq:lasso-kkt}
\begin{align}
\frac{\hat{\rho}_t}{t}&=\frac{1}{n} X^T(y-X{\hat{\beta}}_t),\label{eq:lasso-kkta}\\
\hat{\rho}_t &\in \partial\|\hat{\beta}_t\|_1,
\end{align}
\end{subequations}
which are obtained by taking the subdifferential of the objective in \eqref{lasso}.

It is well-known that LASSO solution $\hat{\beta}_t$ is biased \cite{FanLi01}. For example, considering the simple case that $n=p=1$, $X$ is the identity and $y\ge0$, then \eqref{eq:lasso-kkt} yields
\begin{equation}\label{eq:lasso1}
\hat{\beta}_t=\left\{
\begin{array}{lr}
0,& \mbox{if $t<1/y$}; \\
y-1/t,& \mbox{otherwise},
\end{array}
\right.
\end{equation}
while \eqref{eq:Bregman ISS} has the solution
\begin{equation}
\beta_t=\left\{
\begin{array}{lr}
0,& \mbox{if $t<1/y$}; \\
y,& \mbox{otherwise},
\end{array}
\right.
\end{equation}
which is unbiased for $t \ge 1/y$ as $\E[\beta_t]=\E[y]=\beta^*$.

Moreover, the Linearized Bregman ISS \eqref{eq:lb-iss} has the solution,
\begin{equation}
\beta_t=\left\{
\begin{array}{lr}
0,& \mbox{if $t<1/y$}; \\
y(1 - e^{-\kappa(t-1/y)}),& \mbox{otherwise},
\end{array}
\right.
\end{equation}
which  converges to the unbiased Bregman ISS estimator exponentially fast.

Let us discuss this phenomenon in the general setting. First, let the \emph{oracle estimator} be the subset least-squares solution $\tilde{\beta}^*$ given the \emph{true} set of variables $S$ by an oracle, whose nonzero entries are given by
\begin{equation} \label{eq:oracle}
\tilde{\beta}_S^* = \left(\frac{1}{n} X_S^T X_S\right)^{-1} \frac{1}{n} X_S^T y = \beta^*_S +   \left(\frac{1}{n} X_S^T X_S\right)^{-1} \frac{1}{n} X_S^T \epsilon.
\end{equation}
Clearly $\tilde{\beta}^*_S \sim \NN(\beta^*_S, \Sigma_n)$ where
$\Sigma_n=\frac{\sigma^2}{n}\left(\frac{1}{n} X_S^T X_S\right)^{-1}$. Since in expectation with respect to noise, $\E[\tilde{\beta}^*] = \beta^*$, $\tilde{\beta}^*$ is an unbiased estimate of $\beta^*$.

In reality we are not given the support set $S$, so the following two properties are used to evaluate the performance of an estimator $\hat{\beta}$.
\begin{enumerate}
\item \textbf{Model selection consistency}: $\supp(\hat{\beta})=S$;
\item \textbf{Asymptotic normality}: $\sqrt{n} (\hat{\beta}-\beta^\ast) \to \NN(0, \Sigma^*)$, where
$$\Sigma^*=\lim_{n\to \infty}n \Sigma_n =\sigma^2\left( \lim_{n\to\infty} \frac{1}{n} X_S^T X_S\right)^{-1}.$$
%, that has the same covariance structure as does the oracle estimator.
\end{enumerate}
Since these properties hold for the oracle estimator, they are often referred to as the \emph{oracle properties}.

A solution mapping $\hat{\beta}_t: [0,\infty) \to \R^p$ gives a \emph{regularization path}. Model selection consistency, also known as path consistency, refers to the existence of a point  $\hat{\beta}_{\tau}$ on this path that selects the correct variables, namely, $\supp(\hat{\beta}_{\tau}) = S$. Path consistency has been obtained for LASSO by establishing the stronger property of \emph{sign consistency}, that is, $\sign(\hat{\beta}_{\tau})=\sign(\beta^\ast)$, under certain conditions such as those in \cite{ZhaYu06,Zou06,YuaLin07,Wainwright09}. Provided that path consistency is reached at $\tau$, the LASSO estimate $\hat{\beta}_\tau $ is nonetheless biased since
\begin{equation} \label{eq:lasso-est}
\hat{\beta}_{\tau,S} =  \left(\frac{1}{n} X_S^T X_S\right)^{-1} \frac{1}{n} X_S^T y  - \left(\frac{1}{n} X_S^T X_S\right)^{-1} \frac{\hat{\rho}_{\tau}}{\tau},
\end{equation}
where $\rho_\tau=\sign(\hat{\beta}_{\tau}) \in \partial \|\hat{\beta}_\tau\|_1$.
The first-term on the right-hand side equals the oracle estimator $\tilde{\beta}_S^*$, which is unbiased, whereas the second-term never vanishes and is the bias. Hence, the oracle properties are never completely met by LASSO.

The bias can be removed by a simple differentiation of LASSO solution.
To see this, by multiplying $t$ on both sides of \eqref{eq:lasso-kkta} and differentiating it with respect to $t$, any point on the LASSO path satisfies
\begin{equation}\label{eq:110sat} \dot{\hat{\rho}}_t = \frac{1}{n} X^T(y-X(\hat{\beta}_t+t \dot{\hat{\beta}}_t)).
\end{equation}
With path consistency assumed at time $t=\tau$, we have $\beta_{\tau,i}=0,~\forall~i\not\in S$, and from \eqref{eq:110sat} we have
\begin{equation}\label{eq:110sat2} \dot{\hat{\rho}}_{\tau,S} = \frac{1}{n} X_{S}^T(y-X_{S}(\hat{\beta}_{\tau,S}+\tau \dot{\hat{\beta}}_{\tau,S})).
\end{equation}
Generically, sign consistency occurs in a neighborhood and thus $\dot{\rho}_{\tau,S}=0$. Therefore,
\[ \hat{\beta}_{\tau,S} + \tau  \dot{\hat{\beta}}_{\tau,S} =  \left(\frac{1}{n} X_S^T X_S\right)^{-1} \frac{1}{n} X_S^T y =\tilde{\beta}^*_S,\]
which is the oracle estimator without bias! \textbf{This motivates us to replace $(\hat{\beta}_t+t \dot{\hat{\beta}}_t)$ in \eqref{eq:110sat} by just $\beta_t$, which gives the differential inclusions \eqref{eq:Bregman ISSa} of Bregman ISS.} Later we will show that the resulting $\beta_t$ in \eqref{eq:Bregman ISS} indeed reaches sign-consistency under nearly the same condition as LASSO and hence gives the unbiased oracle estimator.

Compared to LASSO, our dynamic approach has great advantages in algorithmic simplicity and estimate quality. In practice, while LASSO is solved for a sequence of regularization parameters (\emph{i.e.} regularization path), with or without extra debiasing steps such as subset least squares, a single run of our algorithms gives the entire path or, in the case of (discrete) linearized Bregman, a (discrete) approximate path. In addition, the linearized Bregman algorithms such as the simple iterative scheme in \eqref{eq:lb2} are readily parallelizable. Such regularization paths can be unbiased, or, in the case of (discrete) linearized Bregman, have less bias than LASSO paths. Generally speaking, our algorithms return regularization paths of improved quality at just a fraction of cost by LASSO.

In addition, \cite{FanLi01} points out that it is \emph{impossible} to achieve \emph{unbiased} estimator with \emph{convex regularization}. To avoid bias in regularized least square problem, it is thus necessary to introduce non-convex penalties (e.g. SCAD \emph{etc.}) which however suffers from the computational difficulty (typically NP-hardness \cite{GJY11,GWYY15}) on locating the global optima. In a contrast, the dynamic approach studied in this paper, without optimizing any objective function, will be seen to play the same role as non-convex regularization but using a new regularization -- early stopping. Such a debiasing ability naturally inherits from the dynamic solution paths, without suffering the cost of finding global optima in nonconvex optimization. % It thus does not suffer the difficulty of finding global optima.

Therefore, in addition to giving the basic solution properties such as existence, uniqueness, and (dis)continuity, we also attempt to explain the good behaviors of the new solution paths and sequence by establishing their \emph{statistical path consistency} property. Basically we argue that
\begin{enumerate}
\item Under nearly the same conditions for LASSO \cite{ZhaYu06,Zou06,YuaLin07,Wainwright09}  that the covariates $x_i$ are sufficiently uncorrelated and the signal $\beta_S^*$ is strong enough, Bregman ISS \eqref{eq:Bregman ISS} with a proper early stopping rule will return the \emph{oracle estimator};
%\item Temporal mean paths of Bregman ISS \eqref{eq:Bregman ISS} are statistically equivalent to LASSO paths;
\item Sign consistency and $l_2$-error bounds of minimax rates can be generalized to the Linearized Bregman iteration \eqref{eq:lb} and its limit dynamics \eqref{eq:lb-iss}, under similar conditions.
\end{enumerate}

\subsection{Notation and assumptions}

%We introduce the following notation and assumptions to $\beta^*$, $X$, and $\epsilon$.
%\begin{itemize}
%\item  Let the true support be denoted by $S=\supp(\beta^*)=\{i:\beta^*_i\not=0\}$, and $T=S^c$ be its complement. Clearly,  $S\cup T=\{1,\ldots,p\}$.
%\item  $X_S$ denotes the submatrix of $X$ formed by the columns of $X$ in $S$, which are assumed to be \emph{linearly independent}. Similarly define $X_T$ so that $[X_S ~X_T]=X$.
%\item Assume $\epsilon\sim \NN(0,\sigma^2 I_n)$. It can generalize to sub-Gaussian without violating most of our results.
%\end{itemize}
Define $\<u,v\> = u^T v$ and $\<u,v\>_n = \frac{1}{n} u^T v$ for $u,v\in \R^n$. Hence $\|u\|_n = \frac{1}{\sqrt{n}}\|u\|$. Let $X^\ast = \frac{1}{n} X^T$ be the adjoint operator of $X$ with respect to inner product $\<\cdot,\cdot\>_n$.
Let the largest and the smallest nonzero magnitudes of $\beta^\ast$ be $\beta^*_{\max}:=\max (|\beta^*_i| : i\in S)$ and $\beta^*_{\min}:=\min (|\beta^*_i| : i\in S)$, respectively.
Similarly define $\tilde{\beta}_{\max}$ and $\tilde{\beta}_{\min}$ for the oracle estimator $\tilde{\beta}^*$ in (\ref{eq:oracle}).
%In all of the following denote
%\[ \textcolor{red}{\sigma_n := \sigma \sqrt{n+2\sqrt{n \log n}}} \]
The dependence of $\rho_t$ and $\beta_t$ (or equivalently $\rho(t)$ and $\beta(t)$) on $t$ is omitted where it is clear from the context. For the reason to be discussed in Section \ref{sec:unique}, we shall assume that $\rho_t$ is right continuously differentiable and $\beta_t$ is right continuous.

Throughout the paper, given two numbers $a$ and $b$, let $a\vee b:= \max(a,b)$.

\subsection{Outline} In the rest of this paper, we establish basic solution properties of Bregman and Linearized Bregman ISS in Section \ref{sec:unique}. Section \ref{sec:consistencyBISS} and Section \ref{sec:general} describe statistical consistency properties of Bregman ISS and their generalizations to Linearized Bregman ISS/discretization, respectively. Section \ref{sec:analysis} is dedicated to the ideas of proofs. Section \ref{sec:adaptive} presents some preliminary data-dependent stopping rules and Section \ref{sec:related} collects some comments on related works. Section \ref{sec:experiment} provides some preliminary numerical results. Conclusions are summarized in Section \ref{sec:conclusion}.

\section{Bregman and Linearized Bregman solution paths} \label{sec:unique}
 It has been pointed out in \cite{burger2013adaptive} that the solution to Bregman ISS \eqref{eq:Bregman ISS} is a piece-wise regularization path given iteratively by the following nonnegative least squares, starting with $k=0$, $t_0=0$, and $\rho_0=\beta_0=0$:
\begin{enumerate}
\item set $t_{k+1}:=\sup\{t>t_k:\rho_{t_k}+\frac{t-t_k}{n} X^T(y-X\beta_{t_k})\in \partial\|\beta_{t_k}\|_1\}$; if $t_{k+1}=\infty$, then \emph{exit};
\item set $\rho_{t_{k+1}}:=\rho_{t_k}+\frac{t_{k+1}-t_k}{n} X^T(y-X\beta_{t_k})$;
\item set $S_{k+1}:=\{i:|(\rho_{t_{k+1}})_i|= 1\}$ and $T_{k+1}=\{1,\ldots,p\}\setminus S_{k+1}$;
\item set $\beta_{t_{k+1}}$ as any solution to
\begin{equation}\label{iss-sub}
\begin{array}{rrl}
\min_\beta &\|y-X\beta\|_2^2&\\
\mbox{subject to}~&(\rho_{t_{k+1}})_i \beta_i\ge 0&\quad\forall~i\in S_{k+1},\\
&\beta_j=0&\quad\forall~j\in T_{k+1}.
\end{array}
\end{equation}
\item set $k=k+1$ and go  to Step 1.
\end{enumerate}
Hence the solution path to \eqref{eq:Bregman ISS} is established by
\begin{equation} \label{eq:rho_t}
\begin{cases}\rho_t = \rho_{t_k}+\frac{t-t_k}{t_{k+1}-t_k}\rho_{t_{k+1}},\\ \beta_t=\beta_{t_k},\end{cases}\quad t\in[t_k,t_{k+1}).
\end{equation}
where $\rho_t$ is piece-wise linear and $\beta_t$ is piece-wise constant. As the nonnegative least square in \eqref{iss-sub} is a convex quadratic programming with linear constraints, the solution always exists but may not be unique, especially in high dimensional setting $n<p$ which fails the strong convexity in least square. The following theorem presents some general conditions to ensure both the existence and uniqueness of solution path.

\begin{thm}[Solution existence and uniqueness for Bregman ISS] \label{thm:ISSunique}Let $\rho_t$ be right continuously differentiable and $\beta_t$ be right continuous. Then, a solution to \eqref{eq:Bregman ISS} is given by $(\beta_t,\rho_t)$ generated by the above algorithm. Solution $\rho_t$ and $X\beta_t$ are unique. In addition, if the columns $x_i$ of $X$ for $i\in \supp(\beta_t)$ are linearly independent for $t\ge 0$, then $\beta_t$ is also unique.
\end{thm}

In high dimensional setting with $n<p$, as more and more variables are selected in $S_t=\supp(\beta_t)$, linear independence of $X_{S_t}$ will be lost at least when $|S_t|>n$ and so is the uniqueness of $\beta_t$. However, the existence and uniqueness of Linearized Bregman ISS is much simpler as shown in the following theorem.

\begin{thm}[Solution existence and uniqueness for Linearized Bregman ISS] \label{thm:LBISSunique} Let $\rho_t$ be right continuously differentiable and $\beta_t$ be right continuous. Then \eqref{eq:lb-iss} has a unique solution.
\end{thm}

The proofs of these theorems are collected in Appendix.

Provided that sign consistency is met by a point on the path at $t=\tau$, \eqref{iss-sub} returns the oracle estimator $\beta_\tau=\tilde{\beta}^*$ as it is the  least-squares problem subject to only sign constraints. Hence, natural questions are: \emph{what conditions will guarantee sign consistency? And, how to determine  $\tau$}? In the sequel, we are going to provide an answer to these questions. Throughout the remaining of this paper, we assume that $\rho_t$ is right continuously differentiable and $\beta_t$ is right continuous, so the existence and uniqueness of solution paths are guaranteed.

%Early stopping as a regularization rule in Landweber iteration has been  studied in literature \cite{EngHanNeu96,YaoRosCap07}. We shall see in this paper, similar to its continuous dynamics, statistical consistency of early stopping regularization for ISS and Linearized Bregman Iteration can be established.

\section{Consistency of Bregman ISS Dynamics} \label{sec:consistencyBISS}
In this section necessary and sufficient conditions are established for noisy sparse signal recovery with Bregman ISS \eqref{eq:Bregman ISS}.

\subsection{Assumptions}
\begin{itemize}
\item[(A1)] {\bf Restricted Strong Convexity}: there is a $\gamma\in (0,1]$,
$$ X^*_S X_S \geq \gamma I.  $$
\item[(A2)] {\bf Irrepresentable Condition}: there is a $\eta \in (0,1)$,
\[ \left\|X^\ast_T X_S^{\dagger} \right \|_\infty = \left\|\frac{1}{n} X^T_T X_S \left ( \frac{1}{n} X^T_S X_S \right )^{-1} \right \|_\infty \leq 1 - \eta \]
where $X_S^{\dagger}:=X_S \left ( \frac{1}{n} X^T_S X_S \right )^{-1} $.
\end{itemize}

Condition A1 says that the Hessian matrix of the empirical risk $\frac{1}{2n}\|y-X\beta\|^2_2$ restricted on the index set $S\times S$ is strictly positive definitive, so the empirical risk is strongly convex when restricted on the support set $S$. Such a condition is necessary in the sense that once it fails, $X_S$ will be linearly dependent and no unique representation is possible under the basis $X_S$.

Condition A2 says that the absolute row sums of matrix $X^\ast_T X_S^{\dagger}$ are all less than one. It has been proposed independently under a variety of names, e.g. Exact Recovery Condition \cite{Tropp04}, Irrepresentable Condition \cite{ZhaYu06}, among \cite{YuaLin07, Zou06}. Here we adopt the name in \cite{ZhaYu06} as it refers to the fact that the regression coefficients of $X_S$ for response $X_j$ ($j\in T$) all have $\ell_1$-norm less than one, i.e.
\[ \beta_j^\prime = \arg\min_{\beta\in \R^s} \frac{1}{2n}\|X_j  - X_S \beta\|^2 ~\Longrightarrow~ \|\beta_j^\prime\|_1 <1, \]
so in this sense one cannot represent the irrelevant covariates $X_T$ by the relevant ones $X_S$ effectively.

Neither A1 nor A2 can be checked when the support set $S$ of signal is not known. Alternatively we can use a more strict but checkable condition proposed in \cite{DonHuo01}.

\begin{itemize}
\item[(A3)] {\bf Mutual Incoherence Condition}:
\[ \mu := \max_{i,j} \left|\frac{1}{n}\left< X_i, X_j \right>\right| < \frac{1}{(2s-1)}, \ \ \ s=|S|.\]
\end{itemize}

It can be shown \cite{Tropp04,CaiWan11} that once A3 holds, then A1 and A2  simultaneously hold with
\[ \gamma = 1 - \mu (s-1)\] % \frac{\mu (s-1)}{n}\max_i \|X_i\|^2\]
since $(1 - \mu (s-1)) I_S \leq X_S^\ast X_S\leq (1 + \mu (s-1)) I_S$, and
\[ \eta = \frac{1 - \mu(2s-1)}{1-\mu(s-1)}.\]
We note that condition A3 is shown to be sharp in the noisy case in \cite{CaiWanXu10}. With these one can translate all the theoretical results with condition A1 and A2 into condition A3.

With these assumptions, the following stopping time is crucial throughout this section
\begin{equation} \label{eq:taubar1}
\overline{\tau} := \frac{\eta} {2 \sigma} \sqrt{\frac{n}{\log p}} \left (\max_{j\in T} \|X_j\|_n\right )^{-1}.
\end{equation}
In applications, we often normalize the measurement matrix $X$ such that $\|X_j\|_n=1$. So the crucial dependence is $\overline{\tau}\sim \eta\sqrt{n/\log p}/\sigma $, which is equivalent to the optimal choice of LASSO parameters \cite{Wainwright09}.

In the sequel, we shall see that the Irrepresentable condition (A2) is essential to ensure the dynamics of ISS firstly evolves on the signal support set $S$, which will be called \emph{no-false-positive} ($\supp(\beta_t) \subseteq S$ for $t\leq \bar{\tau}$); if in addition the signal is strong enough, then all the signals in $S$ will be identified before the wrong ones show up in the paths, indicated by the \emph{sign-consistency} that $\sign(\beta_{\bar{\tau}})=\sign(\beta^*)$. The latter case is also called as \emph{no-false-negative} in statistics. In this case, an $l_2$ error bound of minimax optimal rates can be achieved.

We will examine two scenarios for establishing these results: the first is an interesting mean Bregman ISS path, which is another biased path, distinct to LASSO, yet qualitatively equivalent; the second is the unbiased Bregman ISS path itself, which meets the consistency results above under nearly the same conditions as LASSO.

\subsection{Mean Bregman ISS Path versus LASSO Path}
As we have seen in Section \ref{sec:motivation} near equation \eqref{eq:110sat}, Bregman ISS \eqref{eq:Bregman ISS} can be derived by differentiating LASSO's KKT conditions. Such a relation can be seen precisely by considering the consistency conditions of LASSO on the following temporal \emph{mean path} of Bregman ISS:
\begin{equation} \label{eq:mean}
\bar{\beta}(t) := \frac{1}{t} \int_0^t \beta (s) d s.
\end{equation}
According to Theorem \ref{thm:ISSunique} and Condition A1, Bregman ISS path $\beta_t$ is unique and thus $\bar{\beta}(t)$ is well defined as long as   $\supp(\beta(s))\subseteq S$, $s\in[0,t)$,  where $S$ is the true support.

A connection between Bregman ISS and LASSO lies in the same condition under which their paths from start to time $t$ are supported within the true support $S$. In addition, the Bregman ISS mean path $\bar{\beta}(t)$ is identical to the LASSO path if the Bregman ISS path is incremental with only adding variables, but without dropping.
%$\supp(\beta(s))\subseteq S$, $s\in[0,t)$, and for all $k\ge 1$ such that $t_k< t$, $S_{k-1}\subset S_k$.
In general, the two paths are distinct.

\begin{thm} \label{thm:nfp}
Let $(\beta_t,\rho_t)$ be either the Bregman ISS path \eqref{eq:Bregman ISS} or the LASSO path \eqref{eq:lasso-kkt} with $\rho(t)\in \partial \|\beta_t\|_1$. Assume that for all $t\leq \tau$,
\begin{equation} \label{eq:nfp}
\| X^\ast_{T} X_S^\dagger \rho_S(t)  + t X^\ast_T P_{T} \epsilon \|_\infty < 1,
\end{equation}
where $P_{S^\perp} = I-P_S = I - X_S^\dagger X_S^*$ is the projection matrix onto $\im(X_{S})^\perp$. Then for all $t\le \tau$,
\begin{enumerate}
\item[A.] the Bregman ISS path, its mean path, and the LASSO path all have supports in $S$;
\item[B.] the mean Bregman ISS path $\bar{\beta}(1/\lambda)$ is piecewise linear with $\lambda = 1/t$; % whose intercept is the Bregman ISS path $\beta(t)$ at $t=1/\lambda$;
\item[C.] if the Bregman ISS path is incremental in the sense that $S_t=\supp(\beta_t)$ satisfies $S_{t}\subseteq S_{t^\prime} \subseteq S$ for all $t\leq t^\prime\leq \tau$, then the mean Bregman ISS path is identical to the LASSO path; but they are distinct in general. %meeting at the points where the Bregman ISS path has the same sign pattern as that of its mean path. %if $\sign(\beta(t))=\sign(\bar{\beta}(t))$, $\bar{\beta}(1/\lambda)$ meets the LASSO solution \eqref{lass}.
\end{enumerate}
\end{thm}
\begin{rem}
In particular in noiseless setting, $\epsilon=0$, (\ref{eq:nfp}) becomes
\[ \| X^\ast_{T} X_S^\dagger \rho_S(t)\|_\infty <1 \]
or dropping $\rho_S(t)$ by
\[ \| X^\ast_{T} X_S^\dagger\|_\infty = \| X^\ast_T X_S (X_S^* X_S)^{-1} \|_\infty <1 \]
which is sufficient and necessary to guarantee that both Bregman ISS, LASSO, and OMP \cite{Tropp04} recovers the sparse signal in noiseless setting; once it is violated there is some $S$-sparse signal for which these methods fail.
\end{rem}

 \begin{proof}[Proof of Theorem \ref{thm:nfp}]
Assume there exists a $\tau\geq 0$, such that for all $t\leq \tau$, solution path $\beta(t)$ satisfies $\supp(\beta(t))\subseteq S$. Then Bregman ISS (\ref{eq:Bregman ISS}) splits into
\begin{subequations}\label{eq:bregman-rho}
\begin{align}
\dot \rho_S & = - X^\ast_S X_S (\beta_S- \beta^*_S) + X_S^\ast \epsilon, \label{eq:rhos}\\
\dot \rho_T  & = - X^\ast_T X_S (\beta_S- \beta^*_S) + X_T^\ast \epsilon. \label{eq:rhot}
\end{align}
\end{subequations}
%\begin{equation} \label{eq:rhos}
%\dot \rho_S= - X^\ast_S X_S (\beta_S- \beta^*_S) + X_S^\ast \epsilon
%\end{equation}
%and
%\begin{equation} \label{eq:rhot}
%\dot \rho_T  = - X^\ast_T X_S (\beta_S- \beta^*_S) + X_T^\ast \epsilon
%\end{equation}
From (\ref{eq:rhos}) one gets the Bregman ISS solution
\begin{equation} \label{eq:ISS_S}
\beta_S(t) =\beta^*_S- ( X^*_S X_S)^{-1}\dot \rho_S  + (X^*_S X_S)^{-1} X_S^\ast \epsilon,
\end{equation}
which leads to the following equation by plugging into (\ref{eq:rhot})
\begin{equation}
\dot \rho_T = X^\ast_T X_S^\dagger \dot \rho_S + X^\ast_T P_T \epsilon.
\end{equation}
Integration on both sides of this equation and setting
\begin{equation} \label{eq:bregman-nfp}
\|\rho_T(t)\|_\infty=\| X^\ast_{T} X_S^\dagger \rho_S(t)  + t X^\ast_T P_{T} \epsilon \|_\infty < 1
\end{equation}
which ensures that $\beta_T(t)=0$. So is the mean path.

On the other hand, LASSO starts from the KKT condition (\ref{eq:lasso-kkt}) which splits into
\begin{subequations}\label{eq:lasso-rho}
\begin{align}
\hat{\rho}_S/t & = - X^\ast_S X_S (\hat{\beta}_S- \beta^*_S) + X_S^\ast \epsilon \label{eq:lasso-rhos}\\
\hat{\rho}_T/t  & = - X^\ast_T X_S (\hat{\beta}_S- \beta^*_S) + X_T^\ast \epsilon \label{eq:lasso-rhot}
\end{align}
\end{subequations}
Following the same trick above one can see the same condition (\ref{eq:bregman-nfp}) is met for LASSO to ensure $\hat{\beta}_T(t)=0$. This finishes the proof of part A.

As to part B, for $t\leq \tau$, the mean path is obtained by integration on (\ref{eq:rhos})
\begin{equation}
\bar{\beta}_S(t) = \frac{1}{t} \int_0^t \beta_S(s) d s = \beta^*_S - \frac{1}{t} (X_S^*X_S)^{-1} \rho_S(t)  +  (X_S^*X_S)^{-1} X_S^\ast \epsilon.
\end{equation}
Equation \eqref{eq:rho_t} implies that $\frac{1}{t}\rho_t = \frac{1}{t} \rho_{t_k}+\frac{1-t_k/t}{t_{k+1}-t_k}\rho_{t_{k+1}}$, which is piecewise linear with respect to $\lambda= 1/t$. %whose intercept $\beta^*_S +  (X_S^*X_S)^{-1} X_S^\ast \epsilon = \beta_S(t)$ by \eqref{eq:ISS_S}.

To see part C, let $S_t=\supp(\beta_t)$ for Bregman ISS. If for all $s\leq t\leq \tau$, $S_s\subseteq S_t\subseteq S$, then similar reasoning as above implies that the Bregman ISS path satisfies
\begin{equation}
\bar{\beta}_{S_t}(t) =  \beta^*_{S_t} - \frac{1}{t} (X_{S_t}^*X_{S_t})^{-1} \rho_{S_t}(t)  +  (X_{S_t}^*X_{S_t})^{-1} X_{S_t}^\ast \epsilon.
\end{equation}
For such incremental processes, $\rho_{S_t}(t) = \sign(\beta_{S_t}(t))=\sign(\bar{\beta}_{S_t}(t))$ which meets the LASSO path equation
 \begin{equation}
 \hat{\beta}_{\hat{S}_t}(t) = \beta^*_{\hat{S}_t} - \frac{1}{t} (X_{\hat{S}_t}^* X_{\hat{S}_t})^{-1} \hat{\rho}_{\hat{S}_t}(t) + (X_{\hat{S}_t}^*X_{\hat{S}_t})^{-1} X_{\hat{S}_t}^*\epsilon,
 \end{equation}
where $\hat{S}_t=\supp(\hat{\beta}_t)$ for LASSO. But such a relation is lost when variable dropping happens.
% which replaces the mean path $\bar{\beta}_S(t)$ by the LASSO path $\hat{\beta}_S(t)$, and subgradient $\rho_S(t)$ by $\hat{\rho}_S(t)$. Therefore, the mean path is \emph{not} the LASSO path, unless $\rho_S(t)=\hat{\rho}_S(t)$ for all $t\leq \tau$ which does not hold in general. %$\rho_S(t)=\sign(\beta_S(t))=\sign(\bar{\beta}_S(t)) $, i.e. the sign pattern of the Bregman ISS path meets that of the mean path.
 \end{proof}

Despite of the difference to the LASSO path, the mean Bregman ISS path may reach statistical model-selection consistency under the same conditions as LASSO.

\begin{thm}[Sign Consistency of Mean Path]
%Let
%$$\overline{\tau} := \frac{\eta} {2 \sigma} \sqrt{\frac{n}{\log p}} \left (\max_{j\in T} \|X_j\|_n\right )^{-1}. $$
Assume that both (A.1) and (A.2) hold. Then the following holds.
\begin{itemize}
\item[A.] \textbf{(No-false-positive)} the mean path has no-false-positive before time $\overline{\tau}$, i.e., $\forall t \leq \overline{\tau}\,\, \supp(\bar{\beta}_t)\subseteq S$, with probability at least $1-\frac{2}{p\sqrt{\pi\log p}}$;
\item[B.] \textbf{(Sign-Consistency)} moreover if the signal is strong enough such that $\beta^*_{\min} > c_1/\bar{\tau}$,
\[  \ c_1 = \left(\frac{\eta }{\sqrt{\gamma}\max_{j\in T} \|X_j\|_n}+\| (X^\ast_S X_S )^{-1}\|_\infty \right), \]
then with probability at least $1-\frac{2}{p\sqrt{\pi\log p}}$, the mean path $\bar{\beta}_{\bar{\tau}}$ has no false-negative, i.e. $\sign(\bar{\beta}_{\bar{\tau}})= \sign(\beta^*)$.
%then there is a $\tau\leq \bar{\tau}$ such that for all $t\in[\tau,\bar{\tau}]$, the mean path $\bar{\beta}(t)$ has no false-negative, i.e. $\sign(\bar{\beta}(t)) = \sign(\beta^*)$, with probability at least $1-\frac{2}{p\sqrt{\pi\log p}}$.
\end{itemize}
\end{thm}

\begin{rem} Under the same conditions as LASSO with $\lambda^*=1/\bar{\tau}$ \cite{Wainwright09}, the mean path $\bar{\beta}$ of Bregman ISS reaches sign-consistency. These conditions are sufficient and necessary in the sense that once violated, there exists an instance such that the probability of failure will be larger than $1/2$ due to noise. %The theorem leads to an $l_\infty$ upper bound on mean path $\|\bar{\beta}(\tau) - \beta^*\|_\infty \leq O(1/\tau)$ which implies $l_2$-bound $\|\bar{\beta}(\tau) - \beta^*\|_2 \leq O(\sqrt{s}/\tau)=O(\sigma \sqrt{s\log p /n})$ ($s=|S|$), a \textbf{minimax optimal} rate up to a logarithmic factor $\log p$ \cite{RasWaiYu11}.
In this sense, the mean path estimator $\bar{\beta}(\bar{\tau})$ is ``statistically equivalent" to the LASSO estimator. % However, differences still exist unless one further reaches the sign-consistency for Bregman ISS path $\sign(\beta(t)) = \sign(\beta^*)$ where the mean path coincides the LASSO path.
\end{rem}

The mean Bregman ISS path geometrically sheds light on why LASSO incurs bias while Bregman ISS can avoid it. The LASSO path, likes the mean Bregman ISS path, involves some kind of averaging that ensures the path continuity but causes bias. The Bregman ISS path is piecewise constant, allows it to be bias-free.

Now we need to answer the following question: \emph{what are conditions to ensure the sign consistency of the Bregman ISS path}?

\subsection{Consistency of Bregman ISS}

The following theorem tells us that under the irrepresentable (incoherence) condition, the Bregman ISS dynamics always evolves in the support of true signals in the early stage; furthermore if the signal is strong enough then the dynamics will pick up all the true variables before selecting any incorrect ones. When such a sign consistency is reached, Bregman ISS returns the oracle estimator which is unbiased.

\begin{thm}[Sign Consistency of Bregman ISS] \label{thm:Bregman ISS}
%Let
%$$\overline{\tau} := \frac{\eta} {2 \sigma} \sqrt{\frac{n}{\log p}} \left (\max_{j\in T} \|X_j\|_n\right )^{-1}. $$
Assume that both (A.1) and (A.2) hold. Then Bregman ISS (\ref{eq:Bregman ISS}) has paths satisfying:
\begin{itemize}
\item[A.] \textbf{(No-false-positive)} the path has no-false-positive before time $\overline{\tau}$, i.e. $ \forall t\leq \overline{\tau}\,\,\supp(\beta_t)\subseteq S$,  with probability at least $1-\frac{2}{p\sqrt{\pi\log p}}$;
\item[B.] \textbf{(Sign-consistency)} moreover if the signal is strong enough such that
\begin{equation}\label{eq:ISS_condition}
\beta^{*}_{\min}\geq\left(\frac{4 \sigma}{\gamma^{1/2}} \vee\frac{8\sigma(2+\log{s})\left (\max_{j\in T} \|X_j\|_n\right )}{\gamma\eta}\right) \sqrt{\frac{\log{p}}{n}}
\end{equation}
Then with probability at least $1-\frac{2}{p\sqrt{\pi\log p}}$, $\sign(\beta_{\bar{\tau}})=\sign(\beta^*)$.
%Then with probability at least $1-\frac{2}{p\sqrt{\pi\log p}}$, there is a stopping time
%\begin{equation} \label{eq:bregman-tau}
%\tau \leq \frac{8+4\log{s}}{\gamma \beta^*_{\min}} \leq \bar{\tau}
%\end{equation}
%such that for all $t \in [\tau,\overline{\tau}]$, we have $\sign(\beta_t)=\sign(\beta^*)$.
\end{itemize}
\end{thm}

\begin{rem}
Once the sign consistency holds, $\beta(t)$ meets the oracle estimator $\tilde{\beta}^*$ which is unbiased and has a $l_2$-error rate $\|\beta(t) - \beta^*\|_2 \leq O(\sigma \sqrt{s\log s /n})$, even better than the $l_2$-error rate $O(\sigma \sqrt{s\log p /n})$ for the biased LASSO estimator which is already \textbf{minimax optimal} up to a logarithmic factor \cite{RasWaiYu11}.
%For Bregman ISS the theorem leads to an upper bound on mean path $\|\bar{\beta}(\tau) - \beta^*\|_\infty \leq O(1/\tau)$ which implies $l_2$-bound $\|\bar{\beta}(\tau) - \beta^*\|_2 \leq O(\sqrt{s}/\tau)=O(\sigma \sqrt{s\log p /n})$, a \textbf{minimax optimal} rate up to a logarithmic factor \cite{RasWaiYu11}.
%
%\begin{itemize}
%%\item[A.] These conditions are sufficient and necessary in the sense that once violated, there exists an instance such that the probability of failure will be larger than $1/2$ due to noise.
%%\item In general, \textbf{partial sign consistency} holds for all $i\in S$ such that $\tau_i <\overline{\tau}$.
%\item[B.] For Bregman ISS the theorem leads to an upper bound on mean path $\|\bar{\beta}(\tau) - \beta^*\|_\infty \leq O(1/\tau)$ which implies $l_2$-bound $\|\bar{\beta}(\tau) - \beta^*\|_2 \leq O(\sqrt{s}/\tau)=O(\sigma \sqrt{s\log p /n})$, a \textbf{minimax optimal} rate up to a logarithmic factor \cite{RasWaiYu11}.
%\end{itemize}
\end{rem}

%Such a bound is even better than the mean path $\bar{\beta}(\tau)$ and LASSO estimators.
To have sign consistency, Theorem \ref{thm:Bregman ISS} makes a \emph{strong signal} condition with a lower bound on $\beta^*_{\min}$. However even without such a strong signal assumption, the minimax optimal $l_2$-error rates can be achieved disregarding sign consistency.

\begin{thm}[Minimax Optimal $l_2$-Error Bound] \label{thm:l2}
 Assume that both (A1) and (A2) hold. There is a $\tau\in [0,\overline{\tau}]$ such that with probability at least $1-\frac{2}{p\sqrt{\pi\log p}}$,
\[ \|\beta_\tau-\beta^{*}\|_{2}\leq \frac{2\sigma}{\eta\gamma} \left(4\max_{j\in T} \|X_j\|_n+ \eta \sqrt{\gamma}\right)\sqrt{\frac{s\log p}{n}}.\]
\end{thm}

The existence of such $\tau$ does not ensure us to find it easily. However one can use $\bar{\tau}$ at a cost of enlarging the constants by a square root of condition number of $\Sigma_S=X^*_S X_S$.
\begin{cor}
Under the same condition of Theorem \ref{thm:l2} and assuming an upper eigenvalue bound $X^*_S X_S\leq \gamma_{\max} I_S $, then the following holds for all $t\in [\tau,\bar{\tau}]$ with probability at least $1-\frac{2}{p\sqrt{\pi\log p}}$
\[  \|\beta_{t}-\beta^{*}\|_{2}\leq \frac{2\sigma\sqrt{\K(X_S^*X_S)}}{\eta \gamma}\left(4\max_{j\in T} \|X_j\|_n+ \eta\sqrt{\gamma}\right)\sqrt{\frac{s\log p}{n}} \]
where $\K(X_S^*X_S)=\gamma_{\max}/\gamma$ is the condition number of $X_S^*X_S$.
\end{cor}

All the results in this subsection follow from the more general results on Linearized Bregman ISS (\ref{eq:lb-iss}) in the next section by taking $\kappa\to\infty$, whose proofs will be summarized in Section \ref{sec:analysis}.

%%%%%%%%%%%%%%%%%%%%%%%%%%%%%%%%%%%%%%%%%%%%%%%%%

\section{Generalizations to Linearized Bregman ISS and Its Discretization} \label{sec:general}
In this section, we state a general consistency result for Linearized Bregman ISS (\ref{eq:lb-iss}) and Linearized Bregman Iterations (\ref{eq:lb}) whose proofs will be given in the next section.

The following new stopping time replaces \eqref{eq:taubar1} throughout this section.
\begin{equation} \label{eq:taubar2}
\overline{\tau} := \frac{(1 - B/(\kappa\eta))\eta} {2 \sigma} \sqrt{\frac{n}{\log p}} \left (\max_{j\in T} \|X_j\|_n\right )^{-1}.
\end{equation}
Clearly when $\kappa\to \infty$, it reduces to \eqref{eq:taubar1}.

\subsection{Consistency of Linearized Bregman ISS}
The following theorem establishes general conditions for statistical consistency of Linearized Bregman ISS (LBISS) \eqref{eq:lb-iss}.
\begin{thm}[Consistency of LBISS]\label{thm:LB-iss}
%Let
%$$\overline{\tau} := \frac{(1 - B/(\kappa\eta))\eta} {2 \sigma} \sqrt{\frac{n}{\log p}} \left (\max_{j\in T} \|X_j\|_n\right )^{-1}. $$
Assume (A1), (A2), and $\kappa$ is big enough such that
$$ \beta^*_{\max} + 2 \sigma \sqrt{\frac{\log p}{\gamma n}} + \frac{\|X\beta^*\|_2 + 2\sigma \sqrt{s\log{n}}}{n\sqrt{\gamma}} \triangleq B \leq \kappa\eta. $$
Then (\ref{eq:lb-iss}) has paths satisfying:
\begin{itemize}
\item[A.] \textbf{(No-false-positive)} the path has no-false-positive before time $\overline{\tau}$ ,i.e., $\forall t \leq \overline{\tau}$, $ \supp(\beta_t)\subseteq S$, with probability at least $1-\frac{2}{p\sqrt{\pi\log p}}-\frac{1}{n\sqrt{\pi\log n}}$;
\item[B.] \textbf{(No-false-negative for Mean Path)} moreover if the signal is strong enough such that $\beta^*_{\min} > c_1/\bar{\tau}$,
\[  \ c_1 = \left(\frac{(1-B/(\kappa\eta))\eta }{\sqrt{\gamma}\max_{j\in T} \|X_j\|_n}+(1+B/\kappa\eta)\| (X^\ast_S X_S )^{-1}\|_\infty \right), \]
then with probability at least $1-\frac{2}{p\sqrt{\pi\log p}}-\frac{1}{n\sqrt{\pi\log n}}$, the mean path $\bar{\beta}(t)$ satisfies $\sign(\bar{\beta}_{\bar{\tau}}) = \sign(\beta^*)$;
%then with probability at least $1-\frac{2}{p\sqrt{\pi\log p}}-\frac{1}{n\sqrt{\pi\log n}}$, there is a $\tau\leq \bar{\tau}$ such that for all $t\in[\tau,\bar{\tau}]$, the mean path $\bar{\beta}(t)$ satisfies $\sign(\bar{\beta}(t)) = \sign(\beta^*)$.
\item[C.] \textbf{(Sign-consistency for LBISS)} moreover if the smallest magnitude $\beta^{*}_{\min}$ is strong enough and $\kappa$ big enough such that
\[\beta^\ast_{\min} \geq \frac{4 \sigma}{\gamma^{1/2}}\sqrt{\frac{\log p}{n}}, \] \[\frac{8+4\log{s}}{\beta^*_{\min}}+\frac{1}{\kappa}\log(\frac{3\|\beta^\ast\|_{2}}{\beta^*_{\min}}) \leq \overline{\tau},\]
then with probability at least $1-\frac{2}{p\sqrt{\pi\log p}}-\frac{1}{n\sqrt{\pi\log n}}$, $\sign(\beta_{\bar{\tau}})=\sign(\beta^*)$;
%then with probability at least $1-\frac{2}{p\sqrt{\pi\log p}}-\frac{1}{n\sqrt{\pi\log n}}$, there is a stopping time $\tau\leq \bar{\tau}$ such that for all $t \in [\tau,\overline{\tau}]$, $\sign(\beta_t)=\sign(\beta^*)$.
\item[D.] \textbf{($l_{2}$-bound)} for some constant $C$ and $\kappa$ large enough to satisfy
\[\frac{4}{C\gamma}\sqrt{\frac{n}{\log p}}+\frac{1}{2\kappa\gamma}(1+\log{\frac{n\|\beta^\ast\|^{2}_{2}+4\sigma^2s\log{p}/\gamma}{C^{2}s\log{p}}}) \leq \overline{\tau},\]
there is a $\tau\in [0,\overline{\tau}]$ such that $\|\beta_\tau-\beta^{*}\|_{2} \leq (C+\frac{2\sigma}{\gamma^{1/2}}) \sqrt{\frac{s\log p}{n}}$ with probability at least $1-\frac{2}{p\sqrt{\pi\log p}}-\frac{1}{n\sqrt{\pi\log n}}$.
\end{itemize}
\end{thm}

\begin{rem}
\begin{itemize}
\item[A.] For sign-consistency of LBISS,
\[ \beta^\ast_{\min} \geq \left(\frac{4 \sigma}{\gamma^{1/2}} \vee \frac{8\sigma(2+\log{s})\left (\max_{j\in T} \|X_j\| \right )}{\gamma\eta}\right)\sqrt{\frac{\log p}{n}} \]
is enough to guarantee the existence of $\kappa$.
\item[B.]  For $l_{2}$-consistency
\[C \geq \frac{8\sigma\left (\max_{j\in T} \|X_j\|_n\right )}{\eta\gamma}\]
is enough to guarantee the existence of $\kappa$.
\item[C.] Taking $\kappa = \infty$, we get the Theorem \ref{thm:Bregman ISS} for Bregman ISS.
\item[D.] An $l_2$-error bound of the same rate for estimator $\beta(\bar{\tau})$ can be established using the monotonicity of $\|X_S(\tilde{\beta}_S^* - \beta_S(t))\|_2$ (see Appendix) for $t \leq \bar{\tau}$,
\begin{eqnarray*}
\|\beta(\bar{\tau}) - \tilde{\beta}^*\|_2 & \leq & \frac{\| X_S (\beta_S(\bar{\tau}) - \tilde{\beta}_S^* ) \|_2}{\sqrt{n{\gamma}}} \leq \frac{\|X_S(\beta_S(\tau) - \tilde{\beta}_S^* ) \|_2}{\sqrt{n\gamma}} , \ \ \tau\leq \bar{\tau}\\
&\leq &\sqrt{\K(X_S^* X_S) } \left(C+\frac{2\sigma}{\sqrt{\gamma}}\right)\sqrt{\frac{s \log p}{n}},
\end{eqnarray*}
where $\K(X_S^*X_S)$ is the condition number of $X_S^*X_S$.
\end{itemize}
\end{rem}

\subsection{Consistency of Linearized Bregman iterations} The following theorem establishes statistical consistency conditions for  Linearized Bregman Iteration \eqref{eq:lb}.

\begin{thm}[Consistency of Linearized Bregman Iterations]\label{thm:LB}
Let $t_n = \sum_{k=0}^{n-1} \alpha_k$.
%and
%$$\overline{\tau} := \frac{(1 - B/(\kappa\eta))\eta} {2 \sigma} \sqrt{\frac{n}{\log p}} \left (\max_{j\in T} \|X_j\|_n\right )^{-1} .$$
Assume (A1), (A2), and $\kappa$ is big enough such that
$$ \beta^*_{\max} + 2 \sigma \sqrt{\frac{\log p}{\gamma n}}  + \frac{\|X\beta^*\|_2 + 2\sqrt{s\log{n}}}{n\sqrt{\gamma}} \triangleq B \leq \kappa\eta, $$
and step size $\alpha$ is small such that $\kappa \alpha\|X_S X_S^*\|<2$. Then any solution path of (\ref{eq:lb-iss}) satisfies
\begin{itemize}
\item[A.] \textbf{(No-false-positive)} for all $n$~ s.t. $t_n\leq \overline{\tau}$, the path has no-false-positive with probability at least $1-\frac{2}{p\sqrt{\pi\log p}}-\frac{1}{n\sqrt{\pi\log n}}$, $\supp(\beta_k)\subseteq S$;
\item[B.] \textbf{(Sign-consistency)} moreover if the smallest magnitude $\beta^{*}_{\min}$ is strong enough and $\kappa$ is big enough to ensure
\[\beta^\ast_{\min} \geq \frac{4 \sigma}{\gamma^{1/2}}\sqrt{\frac{\log p}{n}}, \]
\[ \frac{8+4\log{s}}{\tilde{\gamma}\beta^*_{\min}}+\frac{1}{\kappa\tilde{\gamma}}\log(\frac{3\|\beta^\ast\|_{2}}{\beta_{\min}}) + 3\alpha \leq \overline{\tau}, \]
where $\tilde{\gamma} := \gamma (1-\kappa \alpha\|X_SX_S^\ast\|/2)$, then with probability at least $1-\frac{2}{p\sqrt{\pi\log p}}-\frac{1}{n\sqrt{\pi\log n}}$, $\sign(\beta_{k^*})=\sign(\beta^*)$ for ${k^*}=\max\{k:t_k\leq\bar{\tau}\}$.
%then with probability at least $1-\frac{2}{p\sqrt{\pi\log p}}-\frac{1}{n\sqrt{\pi\log n}}$, there is a $k^*$ such that $t_{k^*}\leq \bar{\tau}$  satisfying $\sign(\beta_k)=\sign(\beta^*)$ for all $t_k\in [t_{k^*},\bar{\tau}]$.
\item[C.] \textbf{($l_{2}$-bound)} for some large enough constants $\kappa$ and $C$ such that
$$\frac{4}{C\tilde{\gamma}}\sqrt{\frac{n}{\log p}}+\frac{1}{2\kappa\tilde{\gamma}}(1+\log{\frac{n\|\beta^\ast\|^{2}_{2}+4\sigma^2s\log{p}/\gamma}{C^{2}s\log{p}}}) + 2\alpha \leq \overline{\tau},$$
with probability at least $1-\frac{2}{p\sqrt{\pi\log p}}-\frac{1}{n\sqrt{\pi\log n}}$, there is a $k^*$, $t_{k^*}\leq \bar{\tau}$, such that $\|\beta_{k^*}-\beta^{*}\|_{2} \leq (C+\frac{2\sigma}{\gamma^{1/2}}) \sqrt{\frac{s\log p}{n}}$.
\end{itemize}
\end{thm}

\begin{rem}
\begin{itemize}
\item[A.] Taking $\alpha \rightarrow 0$, we have $\tilde{\gamma} = \gamma$, and Theorem \ref{thm:LB-iss} for Linearized Bregman ISS follows.
\item[B.] The condition $\kappa \alpha\|X_SX_S^\ast\|<2$ is necessary to ensure the convergence of LB algorithm in the noiseless case. This condition also guarantees the monotonic descent of $\|X_S(\beta_{S,k}-\tilde{\beta}_S^*)\|$ before $\overline{\tau}$.
\end{itemize}
\end{rem}

\section{Analysis of ISS Dynamics} \label{sec:analysis}

The general idea to analyze differential inclusions in (\ref{eq:Bregman ISS}) and (\ref{eq:lb-iss}) is to associate these dynamics with some potential or Lyapunov functions, which control a fast convergence of solutions to the oracle estimator. When the solution path $\beta(t)$ evolves in the support set $S$, a suitable choice of potential functions should be expected with exponentially fast decay, which enables us to estimate the stopping time of reaching sign consistency and small $l_2$-error.

The difficulty lies in that ISS dynamics are differential inclusions, hence we exploit \emph{differential inequalities} of such a potential function to derive the bounds.

\subsection{Potential function}
One would like to study the dynamics of the following differential inclusion
\begin{subequations}
\begin{align}
 \dot\rho_t+\frac{1}{\kappa}\dot\beta_t&=-X^\ast X(\beta_t-\tilde{\beta}^*)\\
 \rho_t &\in \partial\|\beta_t\|_1,
\end{align}
\end{subequations}
where $\tilde{\beta}^*$ is the oracle estimator. Assuming the right continuity of solutions and multiplying both sides above by $\beta(t) - \tilde{\beta}^*$, one obtains a \emph{potential} or \emph{Lyapunov} function $\Psi:\R^p \to \R^+_0$ associated with the dynamics
\[  \frac{d}{dt}( \Psi(\beta_t) ) = - \frac{1}{n}\|X(\beta_t - \tilde{\beta}^*)\|_2^2,  \]
where
\begin{equation}
\Psi(\beta)= D(\tilde{\beta}^*,\beta)+\frac{\|\beta-\tilde{\beta}^*\|^2}{2\kappa}
\end{equation}
and $ D(\tilde{\beta}^*,\beta)$ is the Bregman distance
\begin{equation}
D_V(\tilde{\beta}^*,\beta):=V(\tilde{\beta}^*)-V(\beta) - \langle \partial V(\beta), \tilde{\beta}^*-\beta\rangle
\end{equation}
induced by the particular convex function $V(\beta) = \|\beta\|_1$. Note that $D(\tilde{\beta}^*,\beta)=\langle \tilde{\beta}^*, \tilde{\rho}-\rho\rangle$ where $\tilde{\rho}\in \partial \|\tilde{\beta}^*\|_1$ and $\rho\in \partial \|\beta\|_1$. Hence sign-consistency $\sign(\beta)=\sign(\tilde{\beta}^*)$ implies that $D(\tilde{\beta}^*,\beta)=0$.

As $n\ll p$, matrix $X$ has a large null-space, and to ensure the stationary point of the dynamics being the oracle solution, one must restrict the dynamics evolving outside the subspace $\ker(X)$.

\subsection{Differential inequality with restricted exponential decay of potential}
Define the following \emph{Oracle Dynamics} as if an oracle discloses the true variable set $S$ such that we restrict our attention on a subspace defined by $S$,
\begin{equation} \label{eq:odyn}
\dot\rho^\prime_S+ \frac{1}{\kappa}   \dot\beta^\prime_S = -X^\ast_S X_S( \beta^\prime_S - \tilde{\beta}^*_S), \ \ \ \rho^\prime_S(t) \in \partial \|\beta^\prime_S(t)\|_1.
\end{equation}
Here $X^\ast_S X_S$ is a $s\times s$ symmetric matrix satisfying the strong convexity $X^\ast_S X_S \geq \gamma I_s$, which will lead to exponentially fast decay of potential function.

To reach this goal, our key treatment here is a differential inequality associated differential inclusion in Oracle Dynamics which is tight enough to ensure the exponential decay of potential function. This is a Bihari's type \cite{Bihari56} nonlinear differential inequality, which generalizes the linear cases of Gr\"{o}nwall-Bellman inequalities \cite{Gronwall19,Bellman43}. In our treatment, a piecewise continuous bound is given which leads to the tight rates in this paper.

\begin{lem}[Generalized Bihari's Inequality] \label{lem:bihari}
The potential $\Psi$ of the Oracle Dynamics above satisfies the following differential inequality
\[ \frac{d}{dt}( \Psi(\beta^\prime_S) )  \leq -\gamma F^{-1}(\Psi(\beta^\prime_S)),\]
where $F^{-1}$ is the right-continuous inverse of the following strictly increasing function
\begin{equation} \label{eq:F}
F(x) = \frac{x}{2\kappa}+ \begin{cases} 0 & 0 \leq x < \tilde{\beta}^{2}_{min} \\
2x/\tilde{\beta}_{min} & \tilde{\beta}_{min}^{2} \leq x \leq s\tilde{\beta}_{min}^{2} \\
2\sqrt{xs} & x \geq s\tilde{\beta}_{min}^{2}.
\end{cases}
\end{equation}
\end{lem}

\begin{rem}
An early analysis of ISS convergence in \cite{BRH07} found the following rate on Bregman distance, $D(\tilde{\beta}^*,\beta_t) \leq O( t^{-1})$. In fact, this rate can be derived from the third bound above, which leads to
\[ \frac{d}{dt} D(\tilde{\beta}^*, \beta_t) \leq - \frac{\gamma}{2s} D(\tilde{\beta}^*,\beta_t)^2,  \ \ \ \kappa\to \infty,\]
when $\beta_t$ evolves on $S$. However such a rate is not fast enough to tight bounds on sign-consistency, which requires $D(\tilde{\beta^*},\beta_t)=0$ for some $t\leq \bar{\tau}$. The key treatment in the piecewise bounds above lies in the second bound, which gives
\[  \frac{d}{dt} D(\tilde{\beta}^*, \beta_t) \leq - \frac{\gamma\tilde{\beta}_{min}}{2} D(\tilde{\beta}^*,\beta_t), \ \ \ \kappa\to \infty,\]
and hence an exponential decay of Bregman distance. As we shall see in the proof, such a fast rate is crucial to ensuring all the strong signals selected before wrong components. Therefore one can achieve the tight stopping rules below for sign-consistency under nearly the same conditions as LASSO.
\end{rem}
%\begin{lem}[Generalized Bihari's Inequality] \label{lem:bihari}
%Consider the differential inclusion
%\[ \frac{d \rho}{d t} + \frac{1}{\kappa}  \frac{d \beta}{d t} = - \textcolor{magenta}{X^\ast_S X_S(} \beta - \tilde{\beta}), \ \ \ \rho(t) \in \partial \|\beta(t)\|_1,\]
%\textcolor{magenta}{where $X_S$ satisfies} $X^\ast_S X_S \geq \gamma I_s$. Define a {\bf potential} (or Lyapunov) function
%\[\Psi(\beta)= D(\tilde{\beta},\beta)+\frac{\|\beta-\tilde{\beta}\|^2}{2\kappa}\]
%where the Bregman divergence of $\|\beta\|_1$ is
%\[ D(\tilde{\beta},\beta):= \|\tilde{\beta}\|_1 - \|\beta\|_1 - \left<\rho, \tilde{\beta} - \beta\right>=\left<\tilde{\rho}-\rho,\tilde{\beta}\right>, \ \ \ \rho\in \partial\|\beta\|_1,\tilde{\rho}\in \partial \|\tilde{\beta}\|_1. \]
%Then the following differential inequality holds
%\[ \frac{d}{dt}( \Psi(\beta) )  \leq -\gamma F^{-1}(\Psi(\beta))\]
%where $F^{-1}$ is the right-continuous inverse of the following map
%\begin{equation} \label{eq:F}
%F(x) = \frac{x}{2\kappa}+ \begin{cases} 0 & 0 \leq x < \tilde{\beta}^{2}_{min} \\
%2x/\tilde{\beta}_{min} & \tilde{\beta}_{min}^{2} \leq x \leq s\tilde{\beta}_{min}^{2} \\
%2\sqrt{xs} & x \geq s\tilde{\beta}_{min}^{2}
%\end{cases}
%\end{equation}
%\end{lem}

%\textcolor{red}{(Plot a figure of function $F$ and its inverse.)}

Such an inequality ensures a decrease of the potential function at a fast enough speed which leads to the following tight estimates on stopping time.

%\subsection{Stopping Time Bounds}

We are concerned with the following stopping time reaching sign-consistency and $l_2$-consistency of Oracle Dynamics, respectively.
Define
\begin{equation}
\tilde{\tau}_{1}:=\inf\{t>0: \sign(\beta^\prime_S)=\sign(\tilde{\beta}_S^*)\},
\end{equation}
\begin{equation}
\tilde{\tau}_{2}(C):=\inf\left\{t>0: ||\beta_S^\prime-\tilde{\beta}_S^*||_{2} \leq C\sqrt{\frac{s\log{p}}{n}}\right\}.
\end{equation}

Equipped with the generalized Bihari's inequality, one can build up the following bounds for stopping time on sign-consistency and $l_2$-consistency, respectively.

\begin{lem} \label{lem:stopping}
The following bounds hold for the Oracle Dynamics (\ref{eq:odyn})
\[\tilde{\tau}_{1} \leq \frac{4+2\log{s}}{\gamma\tilde{\beta}^*_{min}}+\frac{1}{\kappa\gamma}\log(\frac{\|\tilde{\beta}^*\|_{2}}{\tilde{\beta}^*_{min}}), \]
\[\tilde{\tau}_{2}(C) \leq \frac{4}{C\gamma}\sqrt{\frac{n}{\log p}}+\frac{1}{2\kappa\gamma}(1+\log{\frac{n\|\tilde{\beta}^*\|^{2}_{2}}{C^{2}s\log{p}}}).\]
\end{lem}
%
%\begin{lem} \label{lem:stopping}
%Consider the differential inclusion
%\[ \frac{d \rho}{d t} + \frac{1}{\kappa}  \frac{d \beta}{d t} = - \textcolor{magenta}{X^\ast_S X_S(} \beta - \tilde{\beta}), ~ \mbox{where $\textcolor{magenta}{X^\ast_S X_S \geq \gamma I}$}\]
%Define
%\[\tilde{\tau}_{\infty}:=\inf\{t>0: \sign(\beta(t))=\sign(\tilde{\beta})\}\]
%\[\tilde{\tau}_{2}(C):=\inf\{t>0: ||\beta(t)-\tilde{\beta}||_{2} \leq C\sqrt{\frac{s\log{d}}{n}}\}\]
%Then we have the following bounds
%\[\tilde{\tau}_{\infty} \leq \frac{4+2\log{s}}{\gamma\tilde{\beta}_{min}}+\frac{1}{\kappa\gamma}\log(\frac{\|\tilde{\beta}\|_{2}}{\tilde{\beta}_{min}}) \]
%\[\tilde{\tau}_{2}(C) \leq \frac{4}{C\gamma}\sqrt{\frac{n}{\log p}}+\frac{1}{2\kappa\gamma}(1+\log{\frac{n|\tilde{\beta}|^{2}_{2}}{C^{2}s\log{p}}})\]
%\end{lem}

\begin{rem}
\begin{itemize}
\item[A.] $\tilde{\tau}_{1} \leq O(\log s/\beta^*_{\min})$ says that $\beta(t)$ will reach sign-consistency after $t\geq O(\log s/{\tilde{\beta}_{\min}})$. The factor $\log s$ is due to the potential method above which converts a multidimensional dynamics into a one-dimensional differential inequality, and dropping potential exponentially from at least $\|\tilde{\beta}_S^*\|_1 \ge s \tilde{\beta}^*_{\min}$ to $0$ requires necessarily the $O(\log s)$ time.
\item[B.] $\tilde{\tau}_{2}(C) \leq O(\frac{1}{C}\sqrt{\frac{n}{p}})$ says that $l_2$-consistency can be reached before $\overline{\tau} = O(\sqrt{\frac{n}{p}})$ as long as $C$ is a sufficiently large constant.
\end{itemize}
\end{rem}

\subsection{Sign-consistency and $l_2$-error bound}

Now we are ready to reach the sign-consistency and $l_2$-error bound for $\beta(t)$ by setting $\tilde{\tau}_1\leq \bar{\tau}$ and $\tilde{\tau}_2(C)\leq \bar{\tau}$, respectively. In these cases, Oracle Dyanmics (\ref{eq:odyn}) $\beta^\prime_S(t)$ meets the original path $\beta_S(t)$ when restricted on $S$. The complete proofs of Theorem \ref{thm:LB-iss} and its discrete version of Theorem \ref{thm:LB} will be found in Appendix A, together with their supporting lemmas.

%\subsection{Sign-Consistency}
%
%\begin{lem} \label{eq:sign}
%\begin{itemize}
%\item[(A)] For all $t\leq \tau$, solution of (\ref{eq:lb-iss}) $\beta(t)$ contains no false positive if
%\[ \| X^\ast_{T} X_S^\dagger  (\rho_S + \beta_S/\kappa)  + \tau X^\ast_T P_{T} \epsilon \|_\infty < 1 \]
%where $P_T = I - X_S^\dagger X_S^\ast$ is the projection operator onto the column space of $X_T$.
%\item[(B)] mean path $\bar{\beta}(\tau)$ is sign-consistent if
%\[ \sign(\bar{\beta}_S(\tau) )= \sign (\beta^*_S + \Phi_S^{-1} X_S^\ast \epsilon - \frac{1}{\tau} \Phi_S^{-1} (\rho_S (\tau) + \textcolor{red}{\frac{1}{\kappa}} \beta_S(\tau)) = \sign(\beta^\ast_S)  \]
%where $\Phi_S = X^\ast_S X_S = \frac{1}{n}X^T_S X_S$.
%\end{itemize}
%\end{lem}

\section{Data-dependent Stopping Rules for Bregman ISS} \label{sec:adaptive}
All the previous results enable us to select $\bar{\tau}$ as a stopping time which however depends on unknown parameters $\gamma$, $\eta$, and noise level $\sigma$, hence is not a data-dependent stopping rule. In this section we present two preliminary results with early stopping rules comparable to \cite{CaiWan11}, which only depend on the noise level $\sigma$ and thus can be estimated from data. We leave it our future work to explore fully adaptive stopping rules.

In the following, define the residue $r(t):=y-X \beta(t)$. The first theorem adopts the stopping rule based on $\|r(t)\|_2$ and the second theorem is based on $\|Xr(t)\|_\infty$.

\begin{thm}\label{thm:stop1}
Suppose
\[\beta^{*}_{\min}\geq\left(\frac{4 \sigma}{\gamma^{1/2}} \vee\frac{8\sigma(2+\log{s})\left (\max_{j\in T} \|X_j\|_n\right )}{\gamma\eta}\right) \sqrt{\frac{\log{p}}{n}},\]
and
%\[\beta^{*}_{\min}\geq \left( 2\sigma \sqrt{\frac{1+2\sqrt{\log{n}/n}}{\gamma}}\right)\]
\[ \beta^*_{\min} \ge \frac{2\sigma}{\sqrt{\gamma}} \left(\sqrt{1+2\sqrt{\frac{\log n}{n}}}+\sqrt{\frac{\log s}{n}}\right) .\]
Then Bregman ISS with the stopping rule $\|r(t)\|_2 \leq  \sigma \sqrt{n+2\sqrt{n\log{n}}} $ selects the true subset $S$ with probability at least $1-O(1/n)$.
\end{thm}

\begin{rem}
\begin{itemize}
\item This result is comparable to Theorem 7 in \cite{CaiWan11}. %\commyy{For an extension to LB-ISS and LB-iterations, the residue representation has to be careful since it does not fit subset least square.}
\item The first condition on the minimum of magnitude of signals ensures the model selection consistency of the Bregman ISS path and thus indicates that one can find some $t$ along the path so that the residual term satisfies $\|r(t)\|_2 \leq  \sigma \sqrt{n+2\sqrt{n\log{n}}}$. Once the path achieves sign consistency, the Bregman ISS must stop.
\item The second condition $\beta^*_{\min} \ge \frac{2\sigma}{\sqrt{\gamma}} \left(\sqrt{1+2\sqrt{\frac{\log n}{n}}}+\sqrt{\frac{\log s}{n}}\right)$ guarantees that one can not stop earlier before Bregman ISS achieves a full recovery. Note that as $n\to \infty$, one needs $\beta^*_{\min} \geq 2\sigma/\sqrt{\gamma}$ which is a constant.
\end{itemize}
\end{rem}

\begin{thm}\label{thm:stop2}
In addition to \eqref{eq:ISS_condition}, suppose
\[\beta^*_{\min}\ge  \frac{2\sigma \max_i \|X_i\|_n\sqrt{2(1+c)s\log p}}{\sqrt{n}\gamma} + 2\sigma\sqrt{\frac{\log s}{n\gamma}}.\]
%and
%\[ \beta^*_{\min} \ge \frac{2\sigma (\max_i \|X_i\| \sqrt{ s \log p/n\gamma}+\sqrt{\log s})}{\sqrt{n\gamma}}\]
%\[ \beta^*_{\min} \ge \frac{2\sigma}{\gamma} (\max_i \|X_i\|_n +\sqrt{\gamma \log s/\log p})\sqrt{\frac{s \log p}{n}}\]
Then Bregman ISS with the stopping rule $\|X^T r(t)\|_\infty \leq  2\sigma \sqrt{\max_i \|X_i\|\log p} $ ($\delta>0$) selects the true subset $S$ with probability at least $1-O(1/p+1/n)$.
\end{thm}

\begin{rem}
This result is comparable to Theorem 8 in \cite{CaiWan11}, though the lower bound $\beta^*_{\min} \geq O(\sigma \sqrt{s\log p/n})$ loses a factor $\sqrt{s}$ here. As $n\to \infty$, the lower bound can be arbitrarily small.
\end{rem}

The remaining of this section presents the proofs of the theorems above.

\begin{proof}[Proof of Theorem \ref{thm:stop1}]
Lemma 3 in \cite{CaiWan11} or Lemma 5.2 in \cite{CaiXuZha09} shows that with probability at least $1-1/n$, $\epsilon$ is essentially $l_2$ upper bounded
\[  \|\epsilon\|_2\leq \sigma \sqrt{n+2\sqrt{n\log{n}}}. \]
Hence with the same probability,
\[ \|r(\tau^*)\|=\|(I-X_S(X_S^*X_S)^{-1}X_S)\epsilon\|_2 \leq \|\epsilon\|_2 \leq \sigma \sqrt{n+2\sqrt{n\log{n}}} \]
We have now shown that the Bregman ISS stops once the path acheives sign consistency.

Next we are going to show that the algorithm will not stop whenever there is some $i\in S$ such that $\beta_i(t)=0$. By Lemma \ref{lem:residue},
\begin{eqnarray*}
\|r_t\| &\ge &  \| X_S(\tilde{\beta}^*_S -\beta_S(t))  \| \\
& \ge & \sqrt{n \gamma} \| \tilde{\beta}^*_S - \beta_S(t) \| \\
& \ge & \sqrt{n\gamma} \tilde{\beta}^*_{\min} \\
& \ge & 2\sigma \sqrt{n+2\sqrt{n\log{n}}}
\end{eqnarray*}
provided that $\tilde{\beta}^*_{\min} \ge 2 \sigma \sqrt{\frac{1+2\sqrt{\log n/n}}{\gamma}}$. Note that
\[ \| (X_S^* X_S)^{-1} X_S^* \varepsilon \|_\infty \le 2\sigma\sqrt{\frac{\log s}{n\gamma}}, \ \ \ \mbox{w. p. at least $1-2n^{-1}$}, \]
so it suffices to have $\beta^*_{\min} \ge \frac{2\sigma (\sqrt{n+2\sqrt{n\log n}}+\sqrt{\log s})}{\sqrt{n\gamma}}$.
\end{proof}

\begin{proof}[Proof of Theorem \ref{thm:stop2}] By assumptions
\[\beta^{*}_{\min}\geq\left(\frac{4 \sigma}{\gamma^{1/2}} \vee\frac{8\sigma(2+\log{s})\left (\max_{j\in T} \|X_j\|_n\right )}{\gamma\eta}\right) \sqrt{\frac{\log{p}}{n}}.\]
Hence, according to Theorem \ref{thm:LB}, the Bregman ISS achieves the sign consistency with high probability. Assume that at time $\tau^*$, $\beta(\tau^*)$ has the same sign as the underlying sparse signal $\beta$.
For each $t$,
\[ r_t = (I  - X_{S(t)} (X^*_{S(t)} X_{S(t)})^{-1} X^*_{S(t)}) (X_S\beta_S + \epsilon) = s_t+n_t , \]
where $s_t = (I-P_{S(t)})X_S\beta_S$ is the signal part of the residual and $n_t = (I-P_{S(t)})\epsilon$ is the noise part of the residual. Then $r_{\tau^*} = n_{\tau^*}$. Let $b_\infty = \sigma \sqrt{2(1+c)\max_i \|X_i\|\log p}$.
\begin{eqnarray*}
\Pr (\| X^T n_t\|_\infty = \|X^T (I-P_t)\epsilon\|_\infty \ge b_\infty )&\le& \sum_i \Pr ( |X_i^T (I-P_t) \epsilon| \ge b_\infty ) \\
&\le& \sum_i \Pr( |X_i^T \epsilon| \ge b_\infty ) \\
&\le& \frac{2}{p^c\sqrt{2\log p}},
\end{eqnarray*}
which means the algorithm stops at $\tau^*$.

Next we are going to show that the algorithm will not stop whenever there is some $i\in A_t \subseteq S$ such that $\beta_i(t)=0$. By Lemma \ref{lem:residue},
\begin{eqnarray*}
\|X^T r_t\|_\infty &= &  \| X^T [ X_S(\tilde{\beta}^*_S -\beta_S(t)) +(I - P_S) \epsilon  ]\|_\infty, \\
& \ge &  \| X_{S}^T [ X_S(\tilde{\beta}^*_S -\beta_S(t))+ (I - P_S) \epsilon] \|_\infty, \\
& = & \| X_{S}^T X_S(\tilde{\beta}^*_S -\beta_S(t))\|_\infty, \ \ X_{S}^T (I-P_S) \epsilon =0, \\
& \ge & \frac{1}{\sqrt{s}} \| X_{S}^T  X_S(\tilde{\beta}^*_S -\beta_S(t)) \|_2 , \\
& \ge & \frac{n\gamma}{\sqrt{s}}\|\tilde{\beta}^*_S -\beta_S(t)\|_2, \\
& \ge & \frac{n\gamma}{\sqrt{s}} \tilde{\beta}^*_{\min}  \geq b_\infty,
\end{eqnarray*}
provided that $\tilde{\beta}^*_{\min} \ge \frac{\sqrt{s} b_{\infty}}{n \gamma} $. Note that
\[ \| (X_S^* X_S)^{-1} X_S^* \varepsilon \|_\infty \le 2\sigma\sqrt{\frac{\log s}{n\gamma}}, \ \ \ \mbox{w. p. at least $1-2n^{-1}$}, \]
so it suffices to have
\[ \beta^*_{\min} \ge b_\infty +2\sigma\sqrt{\frac{\log s}{n\gamma}} = \frac{\sigma (\max_i \frac{\|X_i\|}{\sqrt{n}} \sqrt{2(1+c) s \log p/\gamma}+\sqrt{\log s})}{\sqrt{n\gamma}} \]
with probability at least $1-O(p^{-1}+n^{-1})$.
%
%\begin{eqnarray*}
%\|X^T r_t\|_\infty &= &  \| X^T ( X_S(\tilde{\beta}^*_S -\beta_S(t)) +(I - P_S) \epsilon  )\|_\infty \\
%& \ge &  \| X_{A_t}^T  X_S(\tilde{\beta}^*_S -\beta_S(t))\|_\infty  - \| X_{A_t}^T (I - P_S) \epsilon  \|_\infty \\
%& \ge & \frac{1}{\sqrt{s-s_t}} \|  X_{A_t}^T  X_S(\tilde{\beta}^*_S -\beta_S(t)) \|_2- b_{\infty} (1 + \max_i \|X_i\| \sqrt{\frac{s}{\gamma n}}) , \ \ \ s_t = |A_t| \\
%& \ge & \frac{n\gamma }{\sqrt{s-s_t}}\|\tilde{\beta}^*_S -\beta_S(t)\|_2- b_{\infty} (1 + \max_i \|X_i\| \sqrt{\frac{s}{\gamma n}})  \\
%& \ge & n\gamma \tilde{\beta}^*_{\min} - b_{\infty} (1 + \max_i \|X_i\| \sqrt{\frac{s}{\gamma n}}) \\
%& \ge & b_\infty
%\end{eqnarray*}
%provided that $\tilde{\beta}^*_{\min} \ge \frac{2 b_{\infty}}{n \gamma}  (1 + \max_i \|X_i\| \sqrt{\frac{s}{\gamma n}})$. Note that
%\[ \| (X_S^* X_S)^{-1} X_S^* \varepsilon \|_\infty \le 2\sigma\sqrt{\frac{\log s}{n\gamma}}, \ \ \ \mbox{w. p. at least $1-2n^{-1}$}, \]
%so it suffices to have
%\begin{eqnarray*}
%\beta^*_{\min} & \ge & \frac{2 b_{\infty}}{n \gamma}  (1 + \max_i \|X_i\| \sqrt{\frac{s}{\gamma n}}) +  2\sigma\sqrt{\frac{\log s}{n\gamma}} \\
%& = & \frac{2\sigma (\sqrt{2(1+\delta) \max_i \|X_i\| \log p/n} (1 + \max_i \|X_i\| \sqrt{\frac{s}{\gamma n}}) +\sqrt{\log s})}{\sqrt{n\gamma}}.
%\end{eqnarray*}
%
\end{proof}

\section{Related work} \label{sec:related}

 \subsection{Regularization and other algorithms}
For general penalized least square problems, \cite{FanLi01} has shown that no convex penalty functions can fully achieve the \emph{oracle properties} and thus one has to resort to non-convex regularization, whose global minimizer is, however, algorithmically difficult to locate. Alternatively, one can apply LASSO for variable selection and then remove the bias in LASSO by solving a subset least squares in the second stage. On the other hand, \cite{OBGXY05} noticed that Bregman iteration may reduce bias, also known as contrast loss, in the context of Total Variation image denoising. In this paper, we shall see that dynamics \eqref{eq:Bregman ISS} can automatically remove bias without any non-convexity or second-stage subset least squares. It is a different kind of regularization via early stopping.

Early stopping regularization has been studied widely in linear inverse problems, e.g. \cite{EngHanNeu96}, and recently in Boosting, e.g. \cite{Friedman01,BuhYu02,YaoRosCap07}. In fact, Linearized Bregman iterations can be viewed as an extension of Landweber iteration (also called $L_2$-Boost in statistics),
\[  \beta_{k+1} = \beta_k  +  \frac{\alpha_k}{n} X^T ( y - X \beta_k),\]
which follows the primal path $\beta_t$ as a gradient descent method solving least square problem. To have solution sparsity, Linearized Bregman iterations \eqref{eq:lb} adds the dual path $\rho_t$ in favor of sparse solutions. For ISS, \cite{BGOX06} notices that early stopping regularization is needed as the Bregman distance between the signal $\beta^*$ and the path $\beta_t$ will first decrease and then increase after the prediction error $\|X\beta - y\|$ drops below the noise level. A further quantification of such early stopping regularization is given in \cite{BRH07} under a source condition.

%To recover a sparse signal, another class of iterative algorithms such as Matching Pursuit (MP) and Orthogonal Matching Pursuit (OMP) \cite{mp,omp,omp1} have been widely used in signal recovery. %In noiseless case, \cite{Tropp04} firstly gives the exact recovery condition and in noisy case, \cite{CaiWan11} shows that OMP with proper early stopping can achieve the oracle properties under the same conditions as LASSO, e.g. in \cite{Wainwright09}. Unfortunately it is not clear how to extend OMP to the matrix completion setting with low-rank sparsity.
Linearized Bregman iteration \eqref{eq:lb2} is shown in \cite{Yin2010} equivalent to the gradient ascent iteration applied to the Lagrange dual of the problem
\begin{equation}
\label{augl1}
\min_\beta \|\beta\|_1 + \frac{1}{2\kappa}\|\beta\|_2^2\quad\mathrm{subject~to}~X\beta = y.
\end{equation}
Such a combined $l_1$ and $l_2$ penalty is called \emph{Elastic Net} in statistics \cite{ZH05}. In particular, $\beta_k$ converges  to the unique solution of \eqref{augl1} at a linear rate (as long as $X\not=0$ and $X\beta = y$ has a solution); see \cite{Lai2013}. In addition, for sufficiently large $\kappa$, the solution to \eqref{augl1} is a solution to the basis pursuit model \cite{bpdn}, which is \eqref{augl1} without $\frac{1}{2\kappa}\|\beta\|_2^2$. In noisy settings, early stopping regularization is necessary for signal recovery. The introduction of Elastic Net in statistics is due to a limitation of LASSO that can select at most $s=n$ variables from $p\gg n$ candidates, where the additional $l_2$-penalty ($\|\beta\|_2^2$) enables one to select $s>n$ variables which might be highly correlated, at the cost of a biased estimator. This scenario is beyond the scope of this paper with the assumption $s\leq n< p$ and is left to be explored in the future. However we note that although the Inverse Scale Space \eqref{eq:lb-iss} can be equivalently viewed as differential inclusions (with a discretization \eqref{eq:lb2}) associated with the Elastic Net penalty, its dynamics does not follow the regularization paths of Elastic Net. The results in this paper basically say that under nearly the same condition as LASSO, Bregman ISS \eqref{eq:Bregman ISS} with early stopping regularization may recover the signal without bias, while the bias in \eqref{eq:lb-iss} and \eqref{eq:lb2} can be controlled to be arbitrarily small by increasing $\kappa$ with the same sign-consistency. Finally, we note that such iterative algorithms can be easily extended to general settings with differentiable convex loss and non-differentiable convex penalty, e.g. Linearized Bregman iteration in matrix completion \cite{CaiCanShe10}.

One should not confuse Linearized Bregman iteration \eqref{eq:lb2} with iterative soft-thresholding algorithm (ISTA), which has appeared under different names in the literature (for example, see \cite{DonJoh95,Donoho98,DeVore98,DemDef02,DauDefDem04,hale2008fixed}),
\[
\beta_{k+1} =\mathrm{shrink}(\beta_k + \frac{\alpha_k}{n} X^T(y-X\beta_k),\lambda_k).
\]
Both have an iterative thresholded dynamics with similar computational costs. However by moving the shrinkage operator to a different place in \eqref{eq:lb2}, Linearized Bregman iteration generates a sparse solution path, while ISTA treats $\lambda_k$ as the regularization parameter and its iterates converge to a LASSO solution with a regularization parameter $\lambda = \lim_{k\to\infty} \lambda_k$. Most ISTA-based LASSO solvers simply use a fixed $\lambda_k$ through out the iteration. Although the others update $\lambda_k$ over the iterations, they do so not aiming to provide a full regularization path but to accelerate convergence; this technique is known as ``continuation'' or homotopy method \cite{hale2008fixed}.

\subsection{Parallel and distributed computing}
It is very easy to implement iteration \eqref{eq:lb2} in parallel and distributed manners and apply it to very large-scale datasets.  Suppose $$  X=[X_1,~X_2,~\ldots,~X_L]\in\R^{n\times p},$$ where $X_\ell$'s are submatrices stored in a distributed manner (on a set of networked workstations). The sizes of $X_\ell$'s are flexible and can be chosen for good load balancing. Let each workstation $\ell$ hold data $y$ and $X_{\ell}$, and variables $z_{k,\ell}$ and $w_{k,\ell}:=X_\ell\beta_{k,\ell}$, which are parts of $z_k$ and summands of $w_k:=X\beta_k$, respectively. The iteration \eqref{eq:lb2} is carried out as
%\begin{subequations}\label{eq:lb2p}
\begin{align*}
\mbox{for}~\ell=1,\ldots,L~\mbox{in parallel:}~&\begin{cases}
z_{k+1,\ell}= z_{k,\ell} +\frac{\alpha_k}{n} X_\ell^T(y-w_k), \\
w_{k+1,\ell}= \frac{1}{\kappa} X_\ell\mathrm{shrink}(z_{k+1,\ell},1),\label{lb2b}
\end{cases}\\
\mbox{all-reduce summation:}~& w_{k+1} = \sum_{\ell=1}^L w_{k+1,\ell},
\end{align*}
%\end{subequations}
where the all-reduce step collects inputs from and then returns the sum to all the $L$\ workstations. It is the sum of $L$ $n$-dimensional vectors, so no matter how the all-reduce step is implemented, the communication cost is independent of $p$. It is important to note that the algorithm is not changed at all. In particular, distributing the data into more computing units, i.e., increasing $L$, does \emph{not} increase the number of iterations. Therefore, the parallel implementation is nearly embarrassingly parallel and truly scalable. In addition, it is also possible to develop implementations for data divided into blocks of rows of $X$ or even smaller subblocks that split both rows and columns. Recently, \eqref{eq:lb2} has also been extended in \cite{yuanlinearized} to a \emph{decentralized} setting where not only data and computation are distributed but communication is restricted to computing units with \emph{direct} communication links so there is no data fusion center or long distance communication. The scheme fits sensor network or multi-party regression over the internet, where long-distance communication incurs long delays and high costs.

%Largely because of its unique features, iteration \eqref{eq:lbreg} has found many applications including compressed sensing \cite{Cai2009a}, image processing \cite{Cai2009b,Cai2009c,Yan2013}, matrix completion \cite{Cai2010}, linear discriminat analysis \cite{Zhang2013}, and split feasibility \cite{Lorenz2013}.

%%%%%%%%%%%%%%%%%%%%%%%%%%%%%%

\section{Experiments} \label{sec:experiment}

%\commyy{In case that we wanna keep this section, we need to do: (1) remove the debiased LASSO in Figure \ref{fig:RegPath}, and only show LASSO, ISS, and convergence of LB to ISS with $\kappa=1,3,9,81$ as $3^k$ and $\alpha = 1/(4\kappa)$; (2) show in table 1 that ISS has better AUC than LASSO, LB with a small $\kappa$ is worse than LASSO but a large $\kappa$ is close to ISS. Remove the prediction based on cross-validation. Define the AUC precisely.}

In this section we provide some experimental results to illustrate the relations among LASSO, Bregman ISS (ISS) and Linearized Bregman iteration (LB). The LASSO paths in comparison are computed by R-package `{\tt lars}', while LB paths are computed with our R-package `{\tt libra}'.

In this experiment we choose $n=80$, $p=100$ and only the first $s=30$ elements of $\beta$ are nonzero ($\beta_j = r_j + \sign(r_j)$, where $r_j\sim \NN(0,1)$, $j=1,\dots,30$). Each sample $x_i$ is drawn from the distribution $\NN(0,\Sigma_p)$. We choose $\Sigma_p=(\sigma_{ij})$, where $\sigma_{ij}=1$ if $i=j$, and $\sigma_{ij} = 1/(3p)$ otherwise. In such a setting, the Irrepresentable (Incoherence) Condition holds with high probability, since $\Sigma_p$ is nearly identity matrix. We choose noise level $\sigma=1$ here,  considering the choice that the magnitude of $\beta_i$ is $O(1)$.

\begin{figure}[!t]
\centering
    \includegraphics[width=0.95\textwidth,height=0.75\textheight]{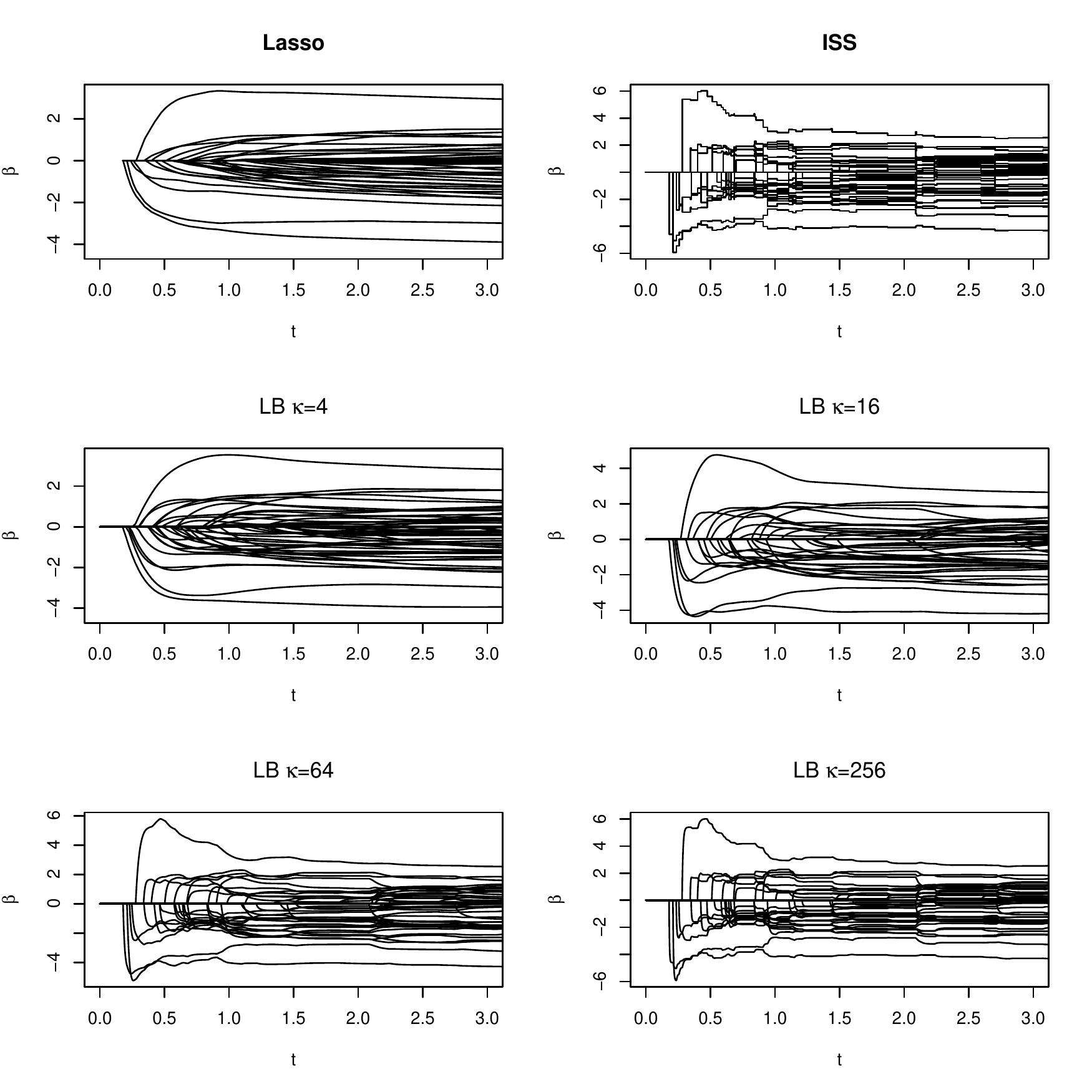}
    \caption{Regularization path of LASSO, Bregman ISS, and Linearized Bregman Iterations with different choices of $\kappa$ ($\kappa\alpha=1/10$). As $\kappa$ grows, the paths of Linearized Bregman iterations approach that of Bregman ISS.}\label{fig:RegPath}
\end{figure}
Figure \ref{fig:RegPath} is an example of regularization path of three methods. %As shown in figure, the paths of ISS and Debiased LASSO are almost the same.
As $\kappa$ goes bigger, the LB path becomes closer to that of ISS. For LB we choose $\kappa\alpha=1/10$ such that the step size of gradient decent is $1/10$, to satisfy the convergence condition. Note that if $\alpha$ is too big, the solution is oscillating.

%Figure \ref{fig:ROC} shows the ROC curves.
%\begin{figure}[!t]
%        \centering
%    \includegraphics[width=0.6\textwidth,height=0.3\textheight]{ROC.eps}
%    \caption{ROC curves. Black circles:LASSO, Black line: ISS, Yellow dotted line: LB with $\kappa=27$. The colored circles are the optimal result of each methods using prediction error as criteria. Red: Lasso, Blue: Debiased Lasso, Green: ISS, Yellow: LB.}\label{fig:ROC}
%\end{figure}

To compare the performance of three methods quantitatively, we choose the AUC of ROC curve, to measure the goodness of three regularization paths. ROC (receiver-operating-characteristic) curve is plotted by thresholding the regularization parameter $\lambda$ in LASSO, $t$ in ISS, or $k$ in LB at different levels which create different true positive rates (TPR) and false positive rates (FPR):
\[TPR = \frac{\#\{Selected~True~Variables\}}{\#\{True~Variables\}}, FPR = \frac{\#\{Selected~False~Variables\}}{\#\{False~Variables\}}.\]

ROC is a curve from $(0,0)$ to $(1,1)$. AUC (Area Under the Curve) means the area under the ROC curve. Large AUC values indicate that the signals are picked out earlier than noise on regularization paths. Repeating the experiments for 100 times, in Table \ref{tab:AUC} we report the mean AUC with standard deviations for the three methods at different noise levels. It shows that all the three methods work reasonably well in this example, while Bregman ISS performs slightly better than LASSO. As $\kappa$ becomes bigger, the performance of LB gets closer to that of Bregman ISS. Notice that as noise level $\sigma$ gets larger, all the methods have their performance decay since signal and noise get confused.
%\begin{table}[!h]
%\centering
%\begin{tabular}{||c||c|c|c||}
%    \hline\hline
%    LASSO & $\sigma=1$ & $\sigma=3$ &$\sigma=5$\\
%    \hline\hline
%    n=200  & 0.9945(0.0068) & 0.9879(0.0086) &0.9671(0.0187)\\
%    \hline\hline
%    ISS & $\sigma=1$ & $\sigma=3$ &$\sigma=5$\\
%    \hline\hline
%    n=200  & 0.9948(0.0064)& 0.9884(0.0082) &0.9676(0.0187)\\
%    \hline
%    LB($\kappa = 4$) & $\sigma=1$ & $\sigma=3$ &$\sigma=5$\\
%    \hline\hline
%    n=200  & 0.9596(0.0136)& 0.9454(0.0173) &0.928(0.0209)\\
%    \hline
%    LB($\kappa = 16$) & $\sigma=1$ & $\sigma=3$ &$\sigma=5$\\
%    \hline\hline
%    n=200  & 0.9745(0.0083)& 0.9664(0.113) &0.9489(0.0173)\\
%    \hline
%    LB($\kappa = 64$) & $\sigma=1$ & $\sigma=3$ &$\sigma=5$\\
%    \hline\hline
%    n=200  & 0.9772(0.0071)& 0.9709(0.0089) &0.9544(0.0159)\\
%    \hline
%\end{tabular}
%\medskip
%\caption{AUC(SD) for three methods with different $\sigma$}
%\end{table}

\begin{table}[!h]
%\centering
\begin{tabular}{||c||c|c|c|c|c||}
\hline\hline
$\sigma$ &LB($\kappa = 4$) & LB($\kappa = 64$) & LB($\kappa = 1024$) &ISS & LASSO\\
\hline
$1$ &0.8747(0.0386) & 0.916(0.0366) & 0.9197(0.0361) & 0.9213(0.0359) & 0.9134(0.0375) \\
\hline
$2$ &0.8604(0.0422) & 0.8931(0.0422) &0.8958(0.0421) & 0.8967(0.0421) & 0.8935(0.0438) \\
\hline
$3$ &0.8306(0.0432) & 0.8513(0.0455) &0.8524(0.0457) & 0.8521(0.0467) & 0.8529(0.0464) \\
\hline
\end{tabular}
\medskip
\caption{Mean AUC (standard deviation) for three methods at different noise levels ($\sigma$): ISS has a slightly better performance than LASSO in terms of AUC and as $\kappa$ increases, the performance of LB approaches that of ISS. As noise level $\sigma$ increases, the performance of all the methods drops. }\label{tab:AUC}
\end{table}

\section{Conclusion and Future Directions} \label{sec:conclusion}
In this paper, noisy sparse signal recovery is approached via dynamics, called Bregman ISS, which can be viewed as a dual gradient descent derived from LASSO KKT conditions. A damped version of this dynamics, Linearized Bregman ISS, can be viewed as a dual gradient descent associated with Elastic Nets. A discretization of Linearized Bregman ISS leads to the widely used Linearized Bregman Iteration algorithm. %As a contrast, the classical LASSO requires a second stage subset least square for debiasing and motivates a nonconvex penalization which increases the difficulty of finding a global optimizer.
Equipped with an early stopping regularization, Bregman ISS can simultaneously achieve model selection consistency and unbiased estimation, under nearly the same conditions as LASSO whose estimators are biased though.
As a discretization of Linearized Bregman ISS paths, model selection consistency and minimax optimal $l_2$-error bounds for Linearized Bregman Iteration are also established. Some data-dependent stopping rules are given for Bregman ISS solution paths.

%Experimental studies confirm our theoretical analysis by showing that Bregman ISS and Linearized Bregman iterations exhibit similar performance to debiased LASSO.

Future directions of our study include fully data-dependent stopping rules and generalization of our results in nonlinear settings.

\appendix

\section{Proofs}\label{app}

\subsection{Proof of Theorem \ref{thm:ISSunique} and \ref{thm:LBISSunique}}

\begin{proof}[Proof of Theorem \ref{thm:ISSunique}]
The existence part follows from \cite{burger2013adaptive}, noticing the nonnegative least squares always have solutions.

We show that the uniqueness part. Define $f(\beta) :=\frac{1}{2n}\|y-X\beta\|^2$. Then, the differential inclusion \eqref{eq:Bregman ISS} is equivalent to
\begin{subequations}\label{eq:gen-iss}
\begin{align}
 \dot\rho_t &= -\nabla f(\beta_t),\label{eq:gen-issa}\\
 \rho_t &\in \partial\|\beta_t\|_1, \label{eq:gen-issb}
\end{align}
\end{subequations}
Let $S^+_t:=\{i:(\rho_t)_i=1\}$, $S^-_t:=\{i:(\rho_t)_i=-1\}$, and $S_t=S_t^+\cup S_t^-$. By \eqref{eq:Bregman ISSb}, in the case of $S_t=\emptyset$, we have $\beta_t=0$, so $-\nabla f(\beta_t)=-\nabla f(0)$ is unique. In the case of $S_t\not=\emptyset$, we show below that $X\beta_t$ and $-\nabla f(\beta_t)$ are  both unique. The uniqueness of $\rho_t$ follows from these results and \eqref{eq:gen-issa}.

In fact,  \eqref{eq:Bregman ISSa} and \eqref{eq:Bregman ISSb} impose the following constraints on $\beta_t$:
\begin{equation}\label{eq:optcond1}  \begin{cases} (\beta_t)_i\ge 0~\mbox{and}~\left(\nabla f(\beta_t)\right)_i\ge 0 ,& （\forall i\in S_t^+,\\
(\beta_t)_i\le 0 ~\mbox{and}~\left(\nabla f(\beta_t)\right)_i\le 0,& （\forall i\in S_t^-,\\
(\beta_t)_i= 0 ,& （\forall i\not\in S_t. \end{cases}
\end{equation}
To see how $\nabla f(\beta_t)$ is involved, notice that $\left(\nabla f(\beta_t)\right)_i\ge 0$ must hold for $\forall i\in S_t^+$ since $(\rho_t)_i\in[-1,1]$ is already at its maximal value 1 and $\nabla f(\beta_t)<0$ is forbidden as it would further increase $(\rho_t)_i$ to an impossible value. The same argument holds for  $\left(\nabla f(\beta_t)\right)_i\le 0$ for $\forall i\in S_t^-$.

Furthermore, we will have $(\beta_t)_i\cdot \left(\nabla f(\beta_t)\right)_i=0$ for all $i$. To see this, assume $(\beta_t)_i\not=0$. Then by the right continuity assumption, there exists an interval $[t,t+\epsilon)$ in which $\beta_i$ remains nonzero with the same sign. By \eqref{eq:gen-issb}, $(\rho_t)_i$ will remain either $+1$ or $-1$ in the same interval, so $\left(\nabla f(\beta_t)\right)_i=0$. On the other hand, assume $\left(\nabla f(\beta_t)\right)_i\not=0$. Then by \eqref{eq:gen-issa}, $\rho_i$ will change and thus it cannot stay either $+1$ or $-1$. By the right continuity of $\beta$, it must hold that $(\beta_t)_i=0$. Therefore, we have the addition constraints
\begin{equation}\label{eq:optcond2} (\beta_t)_i\cdot \left(\nabla f(\beta_t)\right)_i=0. \end{equation}
Conditions \eqref{eq:optcond1} and \eqref{eq:optcond2}  are precisely the KKT optimality conditions for
\begin{align}
\min_\beta~ & f(\beta) \nonumber\\
\mathrm{subject~to}~ &\begin{cases} \beta_i\ge 0,&\forall i\in S_t^+,\\
\beta_i\le 0 ,& \forall i\in S_t^-,\\
\beta_i= 0 ,& \forall i\not\in S_t, \end{cases}\label{eq:optprob-iss}
\end{align}
which is identify to \eqref{iss-sub} except \eqref{iss-sub} specifies the time $t_{k+1}$.
Let $\beta_t$ be the solution to problem \eqref{eq:optprob-iss}.

In general, if $f$ is strictly convex, then the solution $\beta_t$ is unique. In our case, $f$ is not necessarily strictly convex, but  $f = g(X\beta)$ for a strictly convex function $g$. Therefore, $X\beta_t$ is unique, and thus so is $\nabla f(\beta_t) = X^T\nabla g(X\beta_t)$. Lastly, $\beta_t$ is unique if the columns of $X$ corresponding to nonzero entries of $\beta_t$ are linearly independent since $X\beta_t$ is unique.
\end{proof}

\begin{proof}[Proof of Theorem \ref{thm:LBISSunique}]
Let $z_t = f(\rho_t,\beta_t) = \rho_t + \frac{1}{\kappa}\beta_t$, then $f$ is an injective function from the admissible set $(\rho,\beta)$ to $C^1$ in variable $t$ and $\beta_t = \kappa\mathrm{shrink}(z_t,1)$. %$\beta_t = \kappa(z_t - \mathrm{shrink}(z_t,1))$.
Now differential inclusion \eqref{eq:lb-iss} becomes the ODE
\[\dot{z}_t = \frac{1}{n} X^T(y-\kappa X \cdot \mathrm{shrink}(z_t,1) ) =: g(z_t)\]
Obviously, $g(x)$ is Lipschitz continuous. Therefore, the Picard-Lindel\"{o}f Theorem implies that there exists a unique solution to this ODE,  which leads to the solution of \eqref{eq:lb-iss}.
\end{proof}

We note that the solution of \eqref{eq:lb-iss}, though not piece-wise linear or constant, can still be computed in a piece-wise closed form where on each piece, the signs of $\beta_t$ remain unchanged. This is left to the reader.

\subsection{Proof of Consistency of LBISS}

\begin{lem} \label{eq:sign}
\begin{itemize} Assume that $X_S$ has full column rank.
\item[(A)] For all $t\leq \tau$, solution of (\ref{eq:lb-iss}) $\beta(t)$ contains no false positive if
\[ \| X^\ast_{T} X_S^\dagger  (\rho_S + \beta_S/\kappa)  + t X^\ast_T P_{T} \epsilon \|_\infty < 1, \ \ \ \forall t \leq \tau, \]
where $P_T = I - X_S^\dagger X_S^\ast$ is the projection operator onto the column space of $X_T$.
\item[(B)] Mean path $\bar{\beta}(\tau)$ is sign-consistent if
\[ \sign(\bar{\beta}_S(\tau) )= \sign (\beta^*_S + \Phi_S^{-1} X_S^\ast \epsilon - \frac{1}{\tau} \Phi_S^{-1} (\rho_S (\tau) + \frac{1}{\kappa} \beta_S(\tau)) = \sign(\beta^\ast_S)  \]
where $\Phi_S = X^\ast_S X_S = \frac{1}{n}X^T_S X_S$.
\end{itemize}
\end{lem}

No-false-positivity and the sign-consistency for mean path in Theorem \ref{thm:LB-iss}, directly follow this lemma.

\begin{proof}[Proof of Lemma \ref{eq:sign}]
Consider the differential inclusion (\ref{eq:lb-iss})
\[ \dot \rho + \frac{1}{\kappa}  \dot\beta = -\frac{1}{n} X^T (X \beta - y) = - X^\ast X(\beta - \beta^*) + X^\ast \epsilon.\]
Assume there exists a $\tau\geq 0$, such that for all $t\leq\tau$, solution path $\beta(t)$ contains no false-positive, i.e. $\supp(\beta(t))\subseteq S$. Then for all $t\le \tau$,
\begin{equation} \label{eq:ps}
\dot \rho_S + \dot \beta_S/\kappa= - X^\ast_S X_S (\beta_S- \beta^*_S) + X_S^\ast \epsilon,
\end{equation}
and
\begin{equation} \label{eq:pt}
\dot \rho_T +\dot \beta_T/\kappa = - X^\ast_T X_S (\beta_S- \beta^*_S) + X_T^\ast \epsilon.
\end{equation}
From (\ref{eq:ps}) one gets $-(\beta_S - \beta^*_S) = ( X^*_S X_S)^{-1}( \dot \rho_S +  \dot \beta_S/\kappa) - (X^*_S X_S)^{-1} X_S^\ast \epsilon$, which leads to the following equation by plugging into (\ref{eq:pt})
\[ \dot \rho_T + \dot \beta_T/\kappa= X^\ast_T X_S^\dagger  (\dot \rho_S + \dot \beta_S/\kappa) + X^\ast_T P_T \epsilon, \]
where $P_T = I-P_S = I - X_S^\dagger X_S^*$ is the projection matrix onto $\im(X_T)$.

Integration on both sides and setting
\[ \|\rho_T(t) + \beta_T (t)/\kappa\|_\infty=\| X^\ast_{T} X_S^\dagger  (\rho_S(t) + \beta_S(t)/\kappa) + t X^\ast_T P_{T} \epsilon \|_\infty < 1, \]
the first part follows from $\beta_T(t)=\kappa \cdot \mathrm{shrink}(\rho_T(t) + \beta_T (t)/\kappa,1)$.

The second part is obtained by integration on (\ref{eq:ps})
\[ \bar{\beta}_S(\tau) = \frac{1}{\tau} \int_0^\tau \beta_S(t) d t = \beta^*_S - \frac{1}{\tau} \Phi_S^{-1} (\rho_S(\tau)  + \frac{1}{\kappa} \beta_S(\tau)) +  \Phi_S^{-1} X_S^\ast \epsilon, \]
followed by taking $\sign(\bar{\beta}_S(\tau)) = \sign(\beta^\ast_S)$.
\end{proof}

\begin{lem} \label{prob-bound}
Suppose $\epsilon\sim \NN(0,\sigma^2 I_n)$, and $X\in R^{n\times p}$
\begin{eqnarray} \label{eq:prob-bound}
\Pr(\|X^T\epsilon\|_{\infty} > \sigma\sqrt{2(1+\mu)\log p}\max_{j}\|X_j\|) & \le& \frac{1}{p^\mu\sqrt{\pi\log p}} ; \\
\Pr(\|X^T\epsilon\|_{2} > \sigma\sqrt{2(1+\mu)\tr(X^TX)\log p}) &\le& \frac{1}{p^\mu\sqrt{\pi\log p}}.
\end{eqnarray}
\end{lem}

\begin{proof}[Proof of Lemma \ref{prob-bound}]
From the Gaussian tail probability bound,
\begin{eqnarray*}
\Pr(|X_j^T\epsilon| > \sigma\sqrt{2(1+\mu)\log p}\|X_j\|) & \le & 2 \frac{1}{\sqrt{2(1+\mu)\log p}\sqrt{2\pi}}e^{-\frac{2(1+\mu)\log p}{2}} \\
&\leq& \frac{1}{p^{1+\mu}\sqrt{\pi\log p}}.
\end{eqnarray*}
The first inequality is directly the union bound of index $j$. The second inequality is obtained by the fact
\[\{\epsilon:\|X^T\epsilon\|_{2} > \sigma\sqrt{2(1+\mu)\tr(X^TX)\log p}\} \in \bigcup_j \{\epsilon:|X_j^T\epsilon| > \sigma\sqrt{2(1+\mu)\log p}\|X_j\|\},\]
which ends the proof.
\end{proof}

\begin{proof}[Proof of Lemma \ref{lem:bihari}]
Denote
\[A_{t}=\{i \in S |\sign(\tilde{\beta}_{i}^*) \neq \sign(\beta_{i}^\prime)\} \subseteq S.\]
Noticed that
\begin{eqnarray*}
 \|\tilde{\beta}^*_S-\beta_S^\prime\|_{2}^{2} &\geq& \sum_{i \in A_{t}} \tilde{\beta}_{i}^{*2}\\ \nonumber
 &\geq& \max\{\tilde{\beta}_{\min}\sum_{i \in A_{t}} |\tilde{\beta}^*_{i}|, (\sum_{i \in A_{t}} |\tilde{\beta}^*_{i}|)^2/s\} \\ \nonumber
  &\geq& \max\{\tilde{\beta}_{\min}D(\tilde{\beta}^*_S,\beta^\prime_S)/2, D(\tilde{\beta}^*_S,\beta^\prime_S)^2/4s\} ,\nonumber
 \end{eqnarray*}
and
$$\|\tilde{\beta}^*_S-\beta_S^\prime\|_2< \tilde{\beta}_{\min} \Rightarrow A_t=\emptyset\Rightarrow \tilde{\beta}^*_S=\beta_S^\prime \Rightarrow D(\tilde{\beta}^*_S,\beta^\prime_S)=0, $$
then according to the definition of $\Psi$ and $F$, we have
$$\Psi(\beta^\prime_S) = \frac{\|\tilde{\beta}^*_S-\beta_S^\prime\|_2^2}{2\kappa}+D(\tilde{\beta}^*_S,\beta^\prime_S)\leq F(\|\tilde{\beta}^*_S-\beta_S^\prime\|_{2}^{2}) $$ which implies
$$ F^{-1}(\Psi(\beta^\prime_S)) \leq \|\tilde{\beta}^*_S-\beta_S^\prime\|_{2}^{2}. $$
Combining the following result from right continuous differentiability
$$ \left<\frac{d\rho^\prime_S}{dt}, \beta^\prime_S\right> =0, $$
and the strong convexity conditions of $X^*_{s}X_s$, we have
$$
\frac{d}{dt}( \Psi(\beta^\prime_S) ) = -\left<\beta^\prime_S-\tilde{\beta}^*_S,X^*_{s}X_s(\beta^\prime_S-\tilde{\beta}^*_S)\right>  \leq -\gamma \|\tilde{\beta}^*_S-\beta_S^\prime\|_{2}^{2} \leq -\gamma F^{-1}(\Psi(\beta^\prime_S)),
$$
as desired.
\end{proof}

\begin{proof}[Proof of Lemma \ref{lem:stopping}]
From the generalized Bihari's inequality
\[\tilde{\tau}_1\le - \int_0^{\tilde{\tau}_1} \frac{\frac{d}{dt}( \Psi(\beta^\prime_S) )} {\gamma F^{-1}(\Psi(\beta^\prime_S))} dt
 = \frac{1}{\gamma} \int_{\Psi(\tilde{t}_\infty)}^{\Psi(0)} \frac{dx}{F^{-1}(x)}. \]

Note that $\Psi(0) = \|\tilde{\beta}^*_S\|_1 + \frac{\|\tilde{\beta}_S^*\|^2}{2\kappa}$.~so $F^{-1}(\Psi(0))\le \|\tilde{\beta}^*_S\|_2^2$.
By continuity and monotonicity of $F(x)$ on $(\tilde{\beta}_{\min}^2,+\infty)$ and $\Psi(\tilde{t}_1) \ge \frac{\tilde{\beta}^2_{\min}}{2\kappa}$, we have

\begin{eqnarray*}
\gamma\tilde{\tau}_1 &\le& \int_{\frac{\tilde{\beta}^2_{\min}}{2\kappa}}^{\frac{\tilde{\beta}_{\min}^2}{2\kappa} + 2\tilde{\beta}_{\min}} \frac{dx}{F^{-1}(x)} + \int_{\tilde{\beta}_{\min}^2}^{s\tilde{\beta}_{\min}^2} \frac{dF}{x} + \int_{s\tilde{\beta}_{\min}^2}^{\|\tilde{\beta}^*\|_2^2} \frac{dF}{x} \\ \nonumber
&\le&  \int_{\frac{\tilde{\beta}^2_{\min}}{2\kappa}}^{\frac{\tilde{\beta}_{\min}^2}{2\kappa} + 2\tilde{\beta}_{\min}} \frac{dx}{\tilde{\beta}_{\min}^2} + \int_{\tilde{\beta}_{\min}^2}^{s\tilde{\beta}_{\min}^2} (\frac{1}{2\kappa x}+\frac{2}{\tilde{\beta}_{\min}x})dx + \int_{s\tilde{\beta}_{\min}^2}^{\|\tilde{\beta}^*\|_2^2} (\frac{1}{2\kappa x}+\frac{\sqrt{s}}{x^{\frac{3}{2}}}) dx \\ \nonumber
&\leq&\frac{4+2\log{s}}{\tilde{\beta}_{\min}}+\frac{1}{\kappa}\log(\frac{\|\tilde{\beta}^*\|_{2}}{\tilde{\beta}_{\min}}).
\end{eqnarray*}

Proof of $\tilde{\tau}_2$ is straightforward now. For $t<\tilde{\tau}_2$,\\
\[\frac{d \Psi}{dt} \le - \gamma\|\tilde{\beta}^*- \beta \|_{2}^{2} \le - \gamma \frac{C^{2}s\log{d}}{n}.\]
Let
\[\tilde{F}(x) = \frac{x}{2\kappa}+2\sqrt{xs} \geq F(x), \quad  \forall x>0.\]
Let $\tilde{F}^{-1}$ be the right-continuous inverse. Then $\|\beta(t)-\tilde{\beta}^*\|_{2}^{2} \ge F^{-1}(\Psi(\beta)) \ge \tilde{F}^{-1}(\Psi(\beta))$.
By generalized Bihari's inequality
\[- \frac{1}{\gamma}\frac{d \Psi}{dt} \geq \|\tilde{\beta}^*- \beta \|_{2}^{2} \geq \tilde{F}^{-1}(\Psi). \]
Therefore
\[- \frac{1}{\gamma}\frac{d \Psi}{dt} \geq \max\left\{\tilde{F}^{-1}(\Psi), \frac{C^{2}s\log{d}}{n}\right\}. \]
Again, we have
\[\tilde{\tau}_{2}\leq \frac{1}{\gamma}\int_{\Psi(\tilde{t}_{2})}^{\Psi(0)} \frac{dx}{\max\{\tilde{F}^{-1}(x), \frac{C^{2}s\log{d}}{n}\}}.\]
Noticed that
\[ \tilde{F}^{-1}(\Psi(0)) \leq F^{-1}(\Psi(0)) \leq \|\tilde{\beta}^*\|_2^{2}.\]
Therefore,
\begin{eqnarray*}
\gamma \tilde{\tau}_{2}&\leq& \int_{0}^{\tilde{F}(C^{2}s\log{d}/n)} \frac{dx}{  \frac{C^{2}s\log{d}}{n}}+\int_{\tilde{F}(C^{2}s\log{d}/n)}^{\Psi(0)}\frac{dx}{\tilde{F}^{-1}(x)}\\ \nonumber
&\leq& \int_{0}^{(C^{2}s\log{d}/n)/2\kappa + 2Cs\sqrt{\log{d}/n}} \frac{dx}{\frac{C^{2}s\log{d}}{n}}+\int_{C^{2}s\log{d}/n}^{\|\tilde{\beta}^*\|_{2}^{2}}\frac{d\tilde{F}}{x} \\ \nonumber
&\leq& \frac{1}{2\kappa} + \frac{2}{C}\sqrt{\frac{n}{\log{d}}} + \int_{C^{2}s\log{d}/n}^{\|\tilde{\beta}^*\|_{2}^{2}} (\frac{1}{2\kappa x} + \frac{\sqrt{s}}{x^{3/2}})dx\\ \nonumber
&\leq &\frac{4}{C}\sqrt{\frac{n}{\log d}}+\frac{1}{2\kappa}(1+\log{\frac{n\|\tilde{\beta}^*\|^{2}_{2}}{C^{2}s\log{d}}})
\end{eqnarray*}
which gives the bounds.
\end{proof}

\begin{proof}[Proof of Theorem \ref{thm:LB-iss}]
\begin{eqnarray*}
\mathcal{A}&=&\{\epsilon:\|X_S(X_S^\ast X_S)^{-1}X_S^\ast\epsilon\|_2 > 2\sigma\sqrt{s\log n}\}; \\
\mathcal{B}&=&\{\epsilon:\|(X_S^\ast X_S)^{-1}X_S^\ast\epsilon\|_{\infty} > 2\sigma\sqrt{\frac{\log p}{n\gamma}}\};\\
\mathcal{C}&=&\{\epsilon:\|X_T^\ast P_T\epsilon\|_{\infty} > 2\sigma\sqrt{\frac{\log p}{n}}\max_{j\in T}\|X_j\|_n\}.
\end{eqnarray*}
Note $tr(X_S(X_S^\ast X_S)^{-1}X_S^\ast) = s, (X_S^\ast X_S)^{-1}X_S^\ast\cdot X_S(X_S^\ast X_S)^{-1} = (X_S^\ast X_S)^{-1} \preceq 1/\gamma,$ and $X_T^\ast P_T \cdot P_T X_T \preceq X_T^\ast X_T$, using Lemma \ref{prob-bound}, we have
\[\Pr(A)\le \frac{1}{n\sqrt{\pi\log n}}, ~~\Pr(B)\le \frac{1}{p\sqrt{\pi\log p}}, ~~\Pr(C)\le \frac{1}{p\sqrt{\pi\log p}}.\]

\begin{itemize}
\item[(1)] (no-false-positivity for $\beta(t)$ up to $\tau$)
First consider the LB-ISS
\begin{equation}\label{eq:psub}
\frac{d \rho_S}{d t} + \frac{1}{\kappa}  \frac{d \beta_S}{d t} = - X_S^\ast X_S (\beta_S - \tilde{\beta}_S^*)
\end{equation}
where $\tilde{\beta}_S^* = \beta_S^\ast + (X_S^\ast X_S)^{-1}X_S^\ast\epsilon$. It is easy to conclude $\|X_S(\tilde{\beta}_S^*-\beta_S)\|_2$ is monotonically decreasing based on the following observation
$$ \frac{d}{dt}\frac{\|X_S(\tilde{\beta}_S^*-\beta_S)\|^2_2}{2n}=-\left<\frac{d\rho_S}{dt},\frac{d\beta_S}{dt}\right>-\frac{1}{\kappa}\left\|\frac{d\beta_S}{dt}\right\|_2^2=-\frac{1}{\kappa}\left\|\frac{d\beta_S}{dt}\right\|_2^2\leq 0$$
using $\left< d \rho_S(t)/d t, d \beta_S(t)/d t\right> =0$ from the assumption of Bregman ISS paths.
%\commyy{Can we conclude that $\|\tilde{\beta}_S - \beta_S\|_2$ is monotonically nonincreasing by a similar argument?
%\begin{equation}
%\frac{d}{dt}\frac{\|\tilde{\beta}_S-\beta_S\|^2_2}{2n}=-\left<\frac{d\rho_S}{dt},\frac{d\beta_S}{dt}\right>_{(X^*X)^{-1}}-\frac{1}{\kappa}\left\|\frac{d\beta_S}{dt}\right\|_{(X^*X)^{-1}}^2=-\frac{1}{\kappa}\left\|\frac{d\beta_S}{dt}\right\|_2^2\leq 0
%\end{equation}
%using $\left<\frac{d\rho_S}{dt},\frac{d\beta_S}{dt}\right>_{(X^*X)^{-1}} =(\frac{d\rho_S}{dt})^T (X^*X)^{-1} \frac{d\beta_S}{dt} = 0$?}
On the set $\mathcal{A}^c \bigcup \mathcal{B}^c$,
\begin{eqnarray*}
\|\beta_S\|_\infty &\le& \|\tilde{\beta}_S^*\|_\infty + \|\tilde{\beta}_S^* - \beta_S(t)\|_2\\ \nonumber
 &\le& \tilde{\beta}_{\max} + \frac{\|X_S(\tilde{\beta}_S^* - \beta_S(t))\|_2}{\sqrt{n\gamma}} \\
 &\le& \tilde{\beta}_{\max} + \frac{\|X_S\tilde{\beta}_S^*\|_2}{\sqrt{n \gamma}}\\
 &\le& \beta^*_{\max} + 2 \sigma \sqrt{\frac{\log p}{\gamma n}} + \frac{\|X_S\beta^*\|_2 + 2\sigma\sqrt{s\log{n}}}{\sqrt{n\gamma}}.
 \end{eqnarray*}

Denote this upper bound as $B$. Returning to the original problem, by Lemma \ref{eq:sign}, it suffices to have for all $t\leq \tau$,
$$1 >  \| X^\ast_{T} X_S^\dagger  (\rho_S + \beta_S/\kappa)  + t X^\ast_T P_{T} \epsilon \|_\infty .$$
The first part
$$\| X^\ast_{T} X_S^\dagger  (\rho_S+\beta_S/\kappa)\|_\infty \leq (1-\eta)(1+ \|\beta_S\|_\infty/\kappa) \leq 1 - (1 - B/\kappa\eta)\eta$$
 %To see the second part $t \| X^\ast_T P_T \epsilon\|_\infty < (1 - B/\kappa\eta)\eta$, note that the covariance operator $\frac{1}{n^2} X^T_T P_T E(\epsilon \epsilon^T) P_T X_T \leq \frac{\sigma^2}{n} (\frac{1}{n} X^T_T X_T) \leq \frac{\sigma^2}{n} \max_{j\in T} \|X_j\|_n^2$. Concentration inequality ensures that $\|X^\ast_T P_T \epsilon\|_\infty \leq 2 \sigma \sqrt{\log p/n} \max_{j\in T} \|X_j\|_n$, which implies
and for the second part, since
\[ t \leq  \tau := \frac{1 - B/\kappa\eta}{2} \eta \sigma^{-1} \sqrt{n/\log p} \left (\max_{j\in T} \|A_j\|\right )^{-1} = O(\eta \sigma^{-1} \sqrt{n/\log p}), \]
which leads to that on the set $\mathcal{C}^c$, $t \| X^\ast_T P_T \epsilon\|_\infty < (1 - B/\kappa\eta)\eta$.

\item[(2)] (no-false-negativity for the mean path) it suffices to ensure
\[ \beta^*_{\min} > \| \Phi_S^{-1} X_S^\ast \epsilon \|_\infty + \| \frac{1}{\tau} \Phi_S^{-1} ( \rho_S + \beta_S/\kappa ) \|_\infty,\]
where $\Phi_S = X_S^\ast X_S$. The second part on the right hand side is $\|\frac{1}{\tau} \Phi_S^{-1} ( \rho_S + \beta_S/\kappa )\|_\infty \leq \frac{1}{\tau}  \| \Phi_S^{-1}\|_\infty(1+B/\kappa)$. The first part is bounded on the set $\mathcal{B}^c$.
%by concentration inequality $$ \| \Phi_S^{-1} X_S^\ast \epsilon \|_\infty \leq 2 \sigma \sqrt{\log p /n} \gamma^{-1/2}.$$

\item[(3)]  ($l_{2}$-error bound) Lemma \ref{lem:stopping} implies if $C > \frac{8\sigma\left (\max_{j\in T} \|X_j\|_n\right )}{\eta\gamma}$, when $\kappa$ is big enough, we have
\begin{eqnarray*}
\tilde{t}_{2}(C) &\leq& \frac{4}{C\gamma}\sqrt{\frac{n}{\log p}}+\frac{1}{2\kappa\gamma}(1+\log{\frac{n\|\tilde{\beta}^*\|^{2}_{2}}{C^{2}s^{2}\log{p}}}) \\ \nonumber
&\leq& \frac{4}{C\gamma}\sqrt{\frac{n}{\log p}}+\frac{1}{2\kappa\gamma}(1+\log{\frac{n\|\beta^\ast\|^{2}_{2}+4\sigma^2s\log{p}/\gamma}{C^{2}s\log{p}}}) \\ \nonumber
 &\leq& \overline{\tau}.
 \end{eqnarray*}
Thus $\exists \tau\in [0,\overline{\tau}]$
\[ \|\beta_S(\tau)-\tilde{\beta}_{S}^*\|_{2} \leq C\sqrt{s\log(p)/n}. \]
Note that with high probability
\[ \|\beta^{*}_{S}-\tilde{\beta}_{S}^*\|_{2} \leq 2\sigma\sqrt{s\log(p)/n}\gamma^{-1/2}. \]

\item[(4)] (Sign Consistency for $\beta_t$) The condition
$$\beta^{*}_{min}\geq \frac{4 \sigma}{\gamma^{1/2}} \sqrt{\frac{\log p}{n}} $$
implies that  $\tilde{\beta}^*$ has the same sign as $\beta^{*}$ as well as $1/2|\beta^{*}_{i}| \leq |\tilde{\beta}^*_{i}| \leq 3/2|\beta^{*}_{i}|$ for each component $i$. Thus sign consistency is reached when $\tilde{t}_{\infty} \leq \overline{\tau}$, or
\begin{eqnarray*}
\frac{4+2\log{s}}{\gamma\tilde{\beta}_{\min}}+\frac{1}{\kappa\gamma}\log(\frac{\|\tilde{\beta}^*\|_{2}}{\tilde{\beta}_{\min}})
&\leq& \frac{8+4\log{s}}{\beta^*_{\min}\gamma}+\frac{1}{\kappa\gamma}\log(\frac{3\|\beta^\ast\|_{2}}{\beta^*_{\min}}) \\ \nonumber
&\leq& \overline{\tau} \nonumber,
 \end{eqnarray*}
which is ensured by $\kappa$ big enough and
\[ \beta^*_{\min} \geq 2 \tilde{\beta}_{\min} \ge \left(\frac{4 \sigma}{\gamma^{1/2}} \vee\frac{8\sigma(2+\log{s})\left (\max_{j\in T} \|X_j\|_n\right )}{\gamma\eta}\right) \sqrt{\frac{\log{p}}{n}}.\]
\end{itemize}
This completes the proof.
\end{proof}

\subsection{Proof of Consistency of Linearized Bregman Iterations}

First of all, we give a discrete version of generalized Bihari's inequality which is useful for Linearized Bregman iterations (\ref{eq:lb}).
%\commyy{the same comment as continuous version above.}
\begin{lem}[Discrete Generalized Bihari's inequality] \label{lem:dbihari} Consider the LB
\[(\rho_{k+1} - \rho_k) + (\beta_{k+1} - \beta_k)/\kappa = - \alpha_k X^\ast_S X_S(\beta_k-\tilde{\beta}^*),\]
where $X^\ast_S X_S \geq \gamma I$.
Let the potential (or Lyapunov) function be
\[\Psi_k= D(\tilde{\beta},\beta_{k})+\frac{||\beta_k-\tilde{\beta}^*||^2}{2\kappa}.\]
Then the following difference inequality holds
\[ \Psi_{k+1} - \Psi_{k} \leq -\alpha_k\gamma (1-\kappa \alpha_k\|X_SX_S^\ast\|/2) F^{-1}(\Psi_k) ,\]
 where $F$ is defined by (\ref{eq:F}).
\end{lem}

\begin{proof}[Proof of Lemma \ref{lem:dbihari}]
Similar to continue case, we have
$$\|\beta_k-\tilde{\beta}^*\|_{2}^{2} \ge F^{-1}(\Psi_k). $$

Since $\ell_1$-norm is homogeneous of degree 1, its subgradient $\rho \in \partial \|\beta\|_1$ satisfies $\left<\rho, \beta\right>=\|\beta\|_1$. Multiplying $\beta_k-\tilde{\beta}^*$ on the both sides of iteration equation, it leads to
\[ \Psi_{k+1} - \Psi_k + (\rho_{k+1}-\rho_k)\beta_k -\|\beta_{k+1}-\beta_k\|^2/2\kappa = -\alpha_k\left<\beta_k-\tilde{\beta}^*,X_S^{*}X_S(\beta_k-\tilde{\beta}^*) \right> \]
Note that for $i \in S$, $(\rho^{(i)}_{k+1}-\rho^{(i)}_k)\beta^{(i)}_{k+1} = |\beta^{(i)}_{k+1}| - \rho^{(i)}_k\beta^{(i)}_{k+1}\ge 0$
%and $(p^{(i)}_{k+1}-p^{(i)}_k)(u^{(i)}_{k+1}-u^{(i)}_{k}) \ge 0$, \\
\begin{eqnarray}\nonumber
&&\|\beta_{k+1}-\beta_k\|^2/\kappa - 2(\rho_{k+1}-\rho_k)\beta_k \\\nonumber
&\le& \|\beta_{k+1}-\beta_k\|^2/\kappa + 2(\rho_{k+1}-\rho_k)(\beta_{k+1}-\beta_k)\\\nonumber
&\le& \|\beta_{k+1}-\beta_k\|^2/\kappa + 2(\rho_{k+1}-\rho_k)(\beta_{k+1}-\beta_k) +\|\rho_{k+1}-\rho_k\|^2\\ \nonumber
 &\le& \kappa\| \rho_{k+1}-\rho_k + (\beta_{k+1}-\beta_k)/\kappa\|^2 \\ \nonumber
 &=& \kappa \alpha_k^2\| X_S^* X_S(\beta_k-\tilde{\beta}^*)\|^2 \nonumber
\end{eqnarray}
\begin{eqnarray}\nonumber
\Psi_{k+1} - \Psi_k &\le& -\frac{\alpha_k}{n}\left<X_S (\beta_k-\tilde{\beta}^*),X_S (\beta_k-\tilde{\beta}^*) \right> + \frac{\alpha_k^2 \kappa }{2n^2} \left<X_S^TX_S(\beta_k-\tilde{\beta}^*),X_S^TX_S(\beta_k-\tilde{\beta}^*)\right>\\ \nonumber
&=& -\frac{\alpha_k}{n} \left<X_S(\beta_k-\tilde{\beta}^*), (I - \kappa\alpha_k X_SX_S^\ast/2) X_S(\beta_k-\tilde{\beta}^*) \right>\\ \nonumber
&\le& -\frac{\alpha_k}{n} (1-\kappa \alpha_k\|X_SX_S^\ast \|/2)) \|X_S(\beta_k-\tilde{\beta}^*)\|^2\\ \nonumber
&\le& -\alpha_k\gamma (1-\kappa \alpha_k\|X_SX_S^\ast\|/2)) \| \beta_k - \tilde{\beta}^*\|^2 \\ \nonumber
& \le &  -\alpha_k\gamma (1-\kappa \alpha_k\|X_SX_S^\ast\|/2)) F^{-1}(\Psi_k)
\end{eqnarray}
which gives the result.
\end{proof}

Next we present a discrete stopping time bound from the inequality above.

\begin{lem}[Discrete Stopping Time Bounds] \label{lem:dstop}
Consider the LB
\[(\rho_{k+1} - \rho_k) + (\beta_{k+1} - \beta_k)/\kappa = - \alpha_k X^\ast_S X_S(\beta_k-\tilde{\beta}),\]
where $X^\ast_S X_S \geq \gamma I$ and $\alpha_k \leq \alpha$, for all $ k> 0$.

Define
\[\tilde{\tau}_{1} := \inf\left\{\sum_{t=0}^{k-1} \alpha_t: \sign(\beta_k)=\sign(\tilde{\beta}^*)\right\}\]
and
\[\tilde{\tau}_{2}(C):=\inf\left\{\sum_{t=0}^{k-1} \alpha_t: ||\beta_{k}-\tilde{\beta}^*||_{2} \leq C\sqrt{\frac{s\log{p}}{n}}\right\}.\]
Then the following bounds hold,
\[\tilde{\tau}_\infty \le \frac{4+2\log{s}}{\tilde{\gamma}\tilde{\beta}_{\min}}+\frac{1}{\kappa\tilde{\gamma}}\log(\frac{\|\tilde{\beta}^*\|_{2}}{\tilde{\beta}_{\min}}) +  3\alpha,\]

\[\tilde{\tau}_{2}(C) \leq \frac{4}{C\tilde{\gamma}}\sqrt{\frac{n}{\log p}}+\frac{1}{2\kappa\tilde{\gamma}}(1+\log{\frac{n\|\tilde{\beta}^*\|^{2}_{2}}{C^{2}s\log{p}}}) + 2\alpha,\]
where $\tilde{\gamma} = \gamma (1-\kappa\alpha\|X_SX_S^\ast\|/2)$.
\end{lem}

\begin{rem}
Taking $\alpha\to 0$, it recovers the stopping time bounds in continuous case, Lemma \ref{lem:stopping}.
\end{rem}

\begin{proof}[Proof of Lemma \ref{lem:dstop}]
Consider
\[\Psi_k= D(\tilde{\beta},\beta_{k})+\frac{\|\beta_k-\tilde{\beta}^*\|^2}{2\kappa}.\]
For a uniform upper bound on step sizes $\alpha_t \leq \alpha$, by the discrete Bihari's inequality in Lemma \ref{lem:dbihari}
\[
\Psi_{k+1} - \Psi_k \le -\alpha_k\tilde{\gamma} F^{-1}(\Psi_k) \leq -\alpha_k\tilde{\gamma} \tilde{F}^{-1}(\Psi_k),
\]
where $\tilde{\gamma} = \gamma (1-\kappa \alpha\|XX^\ast\|/2)$ and $\tilde{F}(x) = \frac{x}{2\kappa}+2\sqrt{xs} \geq F(x)$, $\forall x>0$.

For $k$ such that $\Psi_k \geq 2\tilde{\beta}_{\min}+\tilde{\beta}_{\min}^2/2\kappa$, denote $L_k =  F^{-1}(\Psi_k)$, which is non-increasing. Define $t_{m} = \sum_{t=0}^{m-1} \alpha_t$. Let $n_1 = \sup\{n: L_n > s\tilde{\beta}_{\min}^2\}$, then
\[\tilde{\gamma}\alpha_k \le \frac{F(L_k) - F(L_{k+1})}{L_k},\]
then for $0\le k \le n_1-1$, %use the concavity of $F(x)$ on $(s\tilde{u}_{min}^2,+\infty)$, we have
\[ \frac{F(L_k) - F(L_{k+1})}{L_k} \le (\frac{\log L_k}{2\kappa} - 2\sqrt{\frac{s}{L_k}}) - (\frac{\log L_{k+1}}{2\kappa} - 2\sqrt{\frac{s}{L_{k+1}}}).\]
This is because of
\[\frac{L_k-L_{k+1}}{L_k} \le \log(\frac{L_k}{L_{k+1}}), \]
using $1-x \leq - \log x$ for $x\le 1$,
\[ \frac{\sqrt{L_k} - \sqrt{L_{k+1}}}{L_k}\le \frac{\sqrt{L_k} - \sqrt{L_{k+1}}}{\sqrt{L_k}\sqrt{L_{k+1}}} = \frac{1}{\sqrt{L_{k+1}}} - \frac{1}{\sqrt{L_k}}, \]
and
\begin{eqnarray}\nonumber
\tilde{\gamma}t_{n_1}
&\le & (\frac{\log L_0}{2\kappa} - 2\sqrt{\frac{s}{L_0}}) - (\frac{\log L_{n_1}}{2\kappa} - 2\sqrt{\frac{s}{L_{n_1}}})\\ \nonumber
&\le & (\frac{\log \|\tilde{\beta}^*\|^2}{2\kappa} - 2\sqrt{\frac{s}{\|\tilde{\beta}\|^2}}) - (\frac{\log s\tilde{\beta}_{\min}^2}{2\kappa} - 2\sqrt{\frac{s}{s\tilde{\beta}_{\min}^2}}). \nonumber
\end{eqnarray}
\\
Let $n_2 = \sup\{n: L_n > \tilde{\beta}_{min}^2\}$,
\[\tilde{\gamma}\alpha_k \le \frac{F(L_k) - F(L_{k+1})}{L_k}.\]
Then similarly, we have
\begin{eqnarray}\nonumber
\tilde{\gamma}(t_{n_2}-t_{n_1+1})
&\le & (\frac{1}{2\kappa}+\frac{2}{\tilde{\beta}_{\min}})(\log L_{n_1+1} - \log L_{n_2}) \\ \nonumber
&\le & (\frac{1}{2\kappa}+\frac{2}{\tilde{\beta}_{\min}})(\log s\tilde{\beta}_{\min}^2 - \log \tilde{\beta}_{\min}^2). \nonumber
\end{eqnarray}
Let $n_3 = \sup\{n: \Psi_n > \tilde{\beta}_{\min}^2/2\kappa\}$,
\begin{eqnarray}\nonumber
\tilde{\gamma}(t_{n_3}-t_{n_2+1})
&\le & \sum_{k = n_2+1}^{n_3-1} \frac{\Psi_{k}-\Psi_{k+1}}{\tilde{\beta}_{\min}^2}\\ \nonumber
&\le & \frac{\frac{\tilde{\beta}_{\min}^2}{2\kappa} + 2\tilde{\beta}_{\min}-\frac{\tilde{\beta}^2_{\min}}{2\kappa}}{\tilde{\beta}_{\min}^2} \\ \nonumber
& = & \frac{2}{\tilde{\beta}_{\min}}.\nonumber
\end{eqnarray}
To sum up, we have
\[\tilde{\tau}_1 \le t_{n_3+1} \le \frac{4+2\log{s}}{\tilde{\gamma}\tilde{\beta}_{\min}}+\frac{1}{\kappa\tilde{\gamma}}\log(\frac{\|\tilde{\beta}^*\|_{2}}{\tilde{\beta}_{\min}}) +  3\alpha.\]
Similarly, we have
\[\tilde{\tau}_{2}(C) \leq \frac{4}{C\tilde{\gamma}}\sqrt{\frac{n}{\log d}}+\frac{1}{2\kappa\tilde{\gamma}}(1+\log{\frac{n\|\tilde{\beta}^*\|^{2}_{2}}{C^{2}s^{2}\log{d}}}) + 2\alpha,\]
which ends the proof.
\end{proof}

\begin{proof}[Proof of Theorem \ref{thm:LB}]

The proof is the same to the continue case. The only difference is the decreasing of $\|X(\beta_k-\tilde{\beta}^*)\|_2$ needs the condition $\kappa \alpha\|X_SX_S^\ast\|<2.$

Consider the LB
\[(\rho_{k+1} - \rho_k) + (\beta_{k+1} - \beta_k)/\kappa = - \alpha_k X^\ast_S X_S(\beta_k-\tilde{\beta}^*),\]
where $X^\ast_S X_S \geq \gamma I$.
\begin{eqnarray}\nonumber
&&\|X_S(\beta_{k+1}-\tilde{\beta}^*)\|^2 - \|X_S(\beta_{k}-\tilde{\beta}^*)\|^2 \\\nonumber
&=& \|X_S(\beta_{k+1}-\beta_k)\|^2 + 2(\beta_{k+1}-\beta_k)^TX_S^TX_S(\beta_{k}-\tilde{\beta}^*)\\\nonumber
&=& \|X_S(\beta_{k+1}-\beta_k)\|^2 - 2n/\alpha_k(\beta_{k+1}-\beta_k)^T [(\rho_{k+1} - \rho_k) + (\beta_{k+1} - \beta_k)/\kappa]     \\\nonumber
 &\le&\|X_S(\beta_{k+1}-\beta_k)\|^2 - 2n/\alpha_k(\beta_{k+1}-\beta_k)^T (\beta_{k+1} - \beta_k)/\kappa     \\\nonumber
 &=& n(\beta_{k+1}-\beta_k)^T (X_S^*X_S-2/\alpha_K\kappa) (\beta_{k+1} - \beta_k) \\\nonumber
 &\le& 0,
\end{eqnarray}
where we have used $\|X_SX_S^\ast\| = \|X_S^\ast X_S\|$. Hence $\|X_S(\tilde{\beta}_S^*-\beta_k)\|_2$ is monotonically nonincreasing.
\end{proof}

Note that this implies that $\|r_t\| :=\|y-X\beta_t\|$ is monotonically nonincreasing for all $t \in (0,\bar{\tau})$. The following lemma makes it precise.

\begin{lem}\label{lem:residue}
For $t\in [0,\bar{\tau}]$, the residue admits an orthogonal decomposition
\[ \|r_t\|^2 =\|y-X\beta_t\|^2 = \| X_S(\tilde{\beta}^*_S -\beta_S(t))  \|^2 + \| P_T \varepsilon\|^2 \]
and is monotonically nonincreasing.
\end{lem}

\begin{proof} By Pythagorean Theorem,
\begin{eqnarray*}
\|r_t\|^2 & = &  \|X_S(\beta^*-\beta_t) + \varepsilon\|^2 = \|P_S X_S(\beta^* -\beta_t) + P_S \varepsilon\|^2+ \| (I-P_S) \varepsilon\|^2 \\
& = &  \| X_S(\beta^* -\beta_t) + X_S(X_S^* X_S)^{-1} X_S^* \varepsilon\|^2+ C_{\epsilon,S} \\
& = & \| X_S(\tilde{\beta}^*_S -\beta_S(t))  \|^2  + C_{\epsilon,S}
\end{eqnarray*}
and the conclusion follows from that $\| X_S(\tilde{\beta}_S^* -\beta_S(t))  \|$ is monotonically nonincreasing.
\end{proof}

\section*{Acknowledgements}
We thank Dr. Ming Yan for helps on fast Matlab codes for computing Bregman ISS paths. The research of Stanley Osher was supported in part by NSF grant 1118971 and ONR grant N000141210838. The research of Yuan Yao was supported in part by National Basic Research Program of China under grant 2012CB825501 and 2015CB856000, as well as NSFC grant 61071157 and 11421110001. The research of Wotao Yin was supported in part by NSF grants DMS-1349855 and DMS-1317602 and ARO MURI grant W911NF-09-1-0383.

\begin{supplement}
\sname{Supplement A}
\label{suppA}
\stitle{Matlab Linearized Bregman codes}
\slink[url]{{http://www.math.ucla.edu/\~{}wotaoyin/software.html}}
%\slink[url]{http://www.math.ucla.edu/\~{}wotaoyin/software.html}
%\sdescription{An R-package of }
\end{supplement}

\begin{supplement}
\sname{Supplement B}
\label{suppB}
\stitle{R Package of Linearized Bregman algorithms}
\slink[url]{https://cran.r-project.org/web/packages/Libra/index.html}
%\slink[url]{http://www.math.pku.edu.cn/teachers/yaoy/reference/Libra\_1.1.tar.gz}
%\sdescription{Matlab codes of  algorithms}
\end{supplement}

\bibliographystyle{amsalpha}
\bibliography{../bib/YY_Endnote,../bib/WY_Mendeley}

\end{document}